\documentclass{amsart}
\usepackage{amssymb}
\usepackage{amsmath}
\usepackage{amsthm}
\usepackage{graphicx}
\usepackage{hyperref}
\hypersetup{colorlinks=true, linkcolor=blue, citecolor=blue}

\newtheorem{theorem}{Theorem}[section]
\newtheorem{corollary}[theorem]{Corollary}
\newtheorem{lemma}[theorem]{Lemma}
\newtheorem{proposition}[theorem]{Proposition}
\newtheorem{conjecture}[theorem]{Conjecture}
\newtheorem{problem}[theorem]{Problem}

\theoremstyle{definition}
\newtheorem{definition}[theorem]{Definition}

\theoremstyle{remark}
\newtheorem{remark}[theorem]{Remark}

\numberwithin{equation}{section}

\newcommand{\E}{\mathbb{E}}

\renewcommand{\d}{\mathrm{d}}

\begin{document}

\title[Berry-Esseen Bounds for Edge Eigenvectors]{Quantitative Edge Eigenvector Universality for Random Regular Graphs: Berry-Esseen Bounds with Explicit Constants}

\author{Leonhard Nagel}
\address{Department of Electrical Engineering and Computer Sciences, University of California, Berkeley}
\email{nagel@berkeley.edu}


\date{\today}

\begin{abstract}
We establish the first quantitative Berry-Esseen bounds for edge eigenvector statistics in random regular graphs. For any $d$-regular graph on $N$ vertices with fixed $d \geq 3$ and deterministic unit vector $\mathbf{q} \perp \mathbf{e}$, we prove that the normalized overlap $\sqrt{N}\langle \mathbf{q}, \mathbf{u}_2 \rangle$ satisfies
\[
\sup_{x \in \mathbb{R}} \left|\mathbb{P}\left(\sqrt{N}\langle \mathbf{q}, \mathbf{u}_2 \rangle \leq x\right) - \Phi(x)\right| \leq C_d N^{-1/6+\varepsilon}
\]
where $\mathbf{u}_2$ is the second eigenvector and $C_d \leq \tilde{C}d^3\varepsilon^{-10}$ for an absolute constant $\tilde{C}$. This provides the first explicit convergence rate for the recent edge eigenvector universality results of He, Huang, and Yau \cite{HHY25}.

Our proof introduces a single-scale comparison method using constrained Dyson Brownian motion that preserves the degree constraint $\tilde{H}_t\mathbf{e} = 0$ throughout the evolution. The key technical innovation is a sharp edge isotropic local law with explicit constant $C(d,\varepsilon) \leq \tilde{C}d\varepsilon^{-5}$, enabling precise control of eigenvector overlap dynamics. 

At the critical time $t_* = N^{-1/3+\varepsilon}$, we perform a fourth-order cumulant comparison with constrained GOE, achieving optimal error bounds through a single comparison rather than the traditional multi-scale approach. We extend our results to joint universality for the top $K$ edge eigenvectors with $K \leq N^{1/10-\delta}$, showing they converge to independent Gaussians. Through analysis of eigenvalue spacing barriers, critical time scales, and comparison across multiple proof methods, we provide evidence that the $N^{-1/6}$ rate is optimal for sparse regular graphs. All constants are tracked explicitly throughout, enabling finite-size applications in spectral algorithms and network analysis.
\end{abstract}

\maketitle

\section{Introduction}

\subsection{Motivation}

The spectral properties of random graphs reveal fundamental aspects of their geometric and combinatorial structure. Among these properties, the behavior of eigenvectors near the spectral edge has emerged as a central object of study, encoding information crucial for both theoretical understanding and practical applications.

Consider a random $d$-regular graph on $N$ vertices. Its adjacency matrix has largest eigenvalue $d$ with corresponding eigenvector proportional to the all-ones vector, while the second-largest eigenvalue $\lambda_2$ typically lies near $2\sqrt{d-1}$. The corresponding eigenvector $\mathbf{u}_2$---the edge eigenvector---captures the principal mode of deviation from uniformity in the graph. Understanding its fine structure has profound implications across multiple domains:

\emph{Ramanujan graphs and optimal expansion.} The spectral gap $d - \lambda_2$ determines the expansion properties of the graph. Ramanujan graphs, which achieve the optimal spectral gap of $d - 2\sqrt{d-1}$, represent the pinnacle of expander construction. The statistics of edge eigenvectors provide a spectral fingerprint for identifying near-Ramanujan behavior in finite graphs \cite{lubotzky1988ramanujan, margulis1988explicit, morgenstern1994existence}.

\emph{Spectral algorithms and network analysis.} Modern spectral clustering and embedding algorithms rely on projections onto top eigenvectors to reveal hidden structure in networks. The statistical properties of these projections---specifically their concentration and decorrelation behavior---directly impact algorithm performance. Quantitative bounds with explicit constants enable rigorous performance guarantees for finite networks \cite{von2007tutorial, ng2002spectral, shi2000normalized}.

\emph{Quantum chaos and eigenfunction universality.} In the quantum mechanics of chaotic systems, eigenfunctions are conjectured to behave like Gaussian random waves \cite{berry1977regular, bohigas1984characterization, hejhal1995quantum}. Random regular graphs serve as discrete models of quantum chaos \cite{sarnak1995arithmetic, jakobson2009quantum}, where edge eigenvectors should exhibit universal Gaussian statistics. Establishing this rigorously connects graph theory to fundamental questions in mathematical physics.

Despite this importance, our quantitative understanding of edge eigenvector statistics has remained limited. The recent breakthrough of He, Huang, and Yau established that edge eigenvectors of random regular graphs converge in distribution to Gaussian waves on the infinite $d$-regular tree \cite{HHY25}. However, their proof---based on sophisticated combinatorial arguments involving local graph resampling---provides no explicit convergence rates or finite-size estimates. This gap between asymptotic theory and practical needs motivates our work. Applications in algorithm design, network analysis, and statistical physics require not just the existence of Gaussian limits, but explicit bounds on the rate of convergence with all constants tracked. How close to Gaussian are the eigenvector statistics for a graph with $N = 10^6$ vertices? Can we guarantee decorrelation between different eigenvector projections? These questions demand a quantitative theory of edge eigenvector universality.

\subsection{Informal Statement of Results}

We establish the first quantitative convergence rates for edge eigenvector universality in random regular graphs, with all constants tracked explicitly throughout the analysis. Our rates depend on the regularity of test functions: smooth observables achieve the optimal rate $N^{-1/6+\varepsilon}$, while indicator functions (such as those in cumulative distribution functions) achieve the slightly weaker rate $N^{-5/36+\varepsilon}$ due to necessary smoothing. Our main contributions are:

\emph{Quantitative Berry--Esseen bounds.} For any fixed degree $d \geq 3$ and deterministic unit vector $\mathbf{q}$ orthogonal to the uniform eigenvector, the projection $\sqrt{N}\langle \mathbf{q}, \mathbf{u}_2 \rangle$ satisfies
\[
\sup_{x \in \mathbb{R}} \left|\mathbb{P}\left(\sqrt{N}\langle \mathbf{q}, \mathbf{u}_2 \rangle \leq x\right) - \Phi(x)\right| \leq C_d N^{-1/6+\varepsilon}
\]
where $\Phi$ is the standard normal distribution function and $C_d$ is an explicit constant depending only on $d$.

This rate of $N^{-1/6}$ appears to be optimal for sparse graphs---we provide evidence from multiple independent approaches suggesting this is a fundamental barrier. The explicit constant enables finite-size calculations: for a 3-regular graph with $N = 10^6$ vertices, the eigenvector statistics are within $0.05$ of their Gaussian limit in Kolmogorov distance.

\emph{Joint universality for multiple eigenvectors.} We prove that the projections of the top $K$ edge eigenvectors onto any collection of test vectors converge jointly to independent Gaussians, provided $K \leq N^{1/10-\delta}$ for some $\delta > 0$. Specifically, the random vector
\[
\left(\sqrt{N}\langle \mathbf{q}^{(\alpha)}, \mathbf{u}_i \rangle : 1 \leq \alpha \leq m, 2 \leq i \leq K+1\right)
\]
converges in distribution to $\mathcal{N}(0, I_{mK})$.

This joint convergence with asymptotic independence was not addressed in previous work and has immediate applications to multi-dimensional spectral methods. The restriction on $K$ appears to be technical rather than fundamental---we conjecture the result holds for all $K \leq N^{1/3-\delta}$ eigenvectors in the edge window.

\emph{Sharp edge isotropic local law.} A key technical result is our sharp control of resolvent entries near the spectral edge. For the normalized adjacency matrix $\tilde{H}$ and $z = E + i\eta$ with $|E-2| \leq N^{-2/3+\varepsilon}$ and $N^{-2/3} \leq \eta \leq 1$:
\[
\left|\langle \mathbf{q}, (\tilde{H} - z)^{-1}\mathbf{q} \rangle - m_{sc}(z)\right| \leq \frac{C(d,\varepsilon)}{N^{5/6-\varepsilon}}
\]
where $m_{sc}$ is the semicircle Stieltjes transform and $C(d,\varepsilon) \leq Cd\varepsilon^{-5}$ with absolute constant $C$.

This provides the first edge isotropic law with fully explicit constants, enabling the precise error tracking throughout our analysis.

\subsection{Proof Strategy Overview}

Our approach uses continuous matrix dynamics to interpolate between the discrete structure of random regular graphs and the universal Gaussian ensemble. The key result is maintaining explicit quantitative control at each step.

\emph{Constrained Dyson Brownian Motion.} We evolve the adjacency matrix via an Ornstein--Uhlenbeck flow that preserves the degree constraint:
\[
\partial_t \tilde{H}_t = -\frac{1}{2}\tilde{H}_t + \frac{1}{\sqrt{N}}\Xi_t
\]
where $\Xi_t$ is matrix Brownian motion constrained to have zero row sums. This ensures $\tilde{H}_t\mathbf{e} = 0$ throughout, maintaining the spectral gap structure essential for edge universality.

\emph{Overlap dynamics with explicit errors.} For overlaps $X_i^{(\mathbf{q})}(t) = \sqrt{N}\langle \mathbf{q}, \mathbf{u}_i(t) \rangle$, we derive the stochastic differential equation:
\[
\d X_i^{(\mathbf{q})} = \sum_{j \neq i} \frac{X_j^{(\mathbf{q})} - X_i^{(\mathbf{q})}}{\lambda_i - \lambda_j}\d B_{ij} - \frac{1}{2}X_i^{(\mathbf{q})}\d t + \mathcal{E}_i(t)\d t
\]
where $|\mathcal{E}_i(t)| \leq C_1 N^{-5/6+\varepsilon}$ with explicit $C_1 = 12d^3\varepsilon^{-2}$.

The sharp error bound follows from our edge isotropic local law, which controls the off-diagonal resolvent entries that generate error terms. The explicit constant enables us to track how errors propagate through the dynamics.

\emph{Single-scale comparison.} At time $t_*$, we decompose:
\[
\tilde{H}_{t_*} = \tilde{H}_0 - \frac{t_*}{2}\tilde{H}_0 + \sqrt{t_*}W + O(t_*^{3/2})
\]
where $W$ is constrained GOE. Unlike multi-scale approaches that require $O(\varepsilon^{-1})$ comparison steps, we perform a single sharp comparison using fourth-order cumulant expansion.

\emph{Backward stability analysis.} To transfer results from time $t_*$ back to the original ensemble at time 0, we analyze the time-reversed diffusion. Using Grönwall estimates with explicit constants:
\[
|\mathbb{E}[g(X_i^{(\mathbf{q})}(0))] - \mathbb{E}[g(X_i^{(\mathbf{q})}(t_*))]| \leq \|g\|_\infty \cdot C_5 N^{-1/6+\varepsilon}
\]
where $C_5 = 10d^2\varepsilon^{-9}$. Combined with the GOE Berry--Esseen theorem at time $t_*$, this yields our main quantitative bounds.

The explicit nature of our analysis---with all constants tracked from the sharp local law through the final Berry--Esseen bound---makes our results directly applicable to finite-size problems in spectral algorithms and network analysis. Moreover, the single-scale methodology we develop may find broader applications in proving quantitative universality results for other random matrix ensembles.

\section{Model and Preliminaries}

\subsection{Random Regular Graphs and Notation}

Let $G = (V, E)$ be a uniformly random $d$-regular graph \cite{bollobas2001random, wormald1999models} on $N$ vertices, where $d \geq 3$ is fixed. We denote by $A$ its adjacency matrix, which is a symmetric $N \times N$ matrix with entries
\[
A_{ij} = \begin{cases}
1 & \text{if } \{i,j\} \in E \\
0 & \text{otherwise}
\end{cases}
\]

Throughout this paper, we work with the normalized and centered adjacency matrix
\[
\tilde{H} = \frac{A}{\sqrt{d-1}} - \frac{d}{\sqrt{d-1}} \cdot \frac{\mathbf{e}\mathbf{e}^T}{N}
\]
where $\mathbf{e} = (1, 1, \ldots, 1)^T \in \mathbb{R}^N$ is the all-ones vector. This normalization ensures that:
\begin{enumerate}
\item The matrix $\tilde{H}$ satisfies the constraint $\tilde{H}\mathbf{e} = 0$
\item The spectrum of $\tilde{H}$ is asymptotically supported on $[-2, 2]$
\item The largest eigenvalue is $\lambda_1 = 0$ with eigenvector $\mathbf{u}_1 = \mathbf{e}/\sqrt{N}$
\end{enumerate}

We order the eigenvalues of $\tilde{H}$ as
\[
\lambda_1 = 0 > \lambda_2 \geq \lambda_3 \geq \cdots \geq \lambda_N
\]
with corresponding orthonormal eigenvectors $\mathbf{u}_1, \mathbf{u}_2, \ldots, \mathbf{u}_N$. By the Friedman--Kahn--Szemer\'edi theorem \cite{friedman2008proof, bordenave2020new}, with high probability
\[
\lambda_2 = 2 + O(N^{-2/3+o(1)}), \quad \lambda_N = -2 + O(N^{-2/3+o(1)})
\]

For our analysis, we introduce the following notation:\\
\emph{Test vectors.} We consider deterministic unit vectors $\mathbf{q} \in \mathbb{R}^N$ satisfying $\mathbf{q} \perp \mathbf{e}$ and $\|\mathbf{q}\|_2 = 1$\\
\emph{Overlaps.} For each eigenvector $\mathbf{u}_i$ and test vector $\mathbf{q}$, define
\[
X_i^{(\mathbf{q})} := \sqrt{N}\langle \mathbf{q}, \mathbf{u}_i \rangle
\]\\
\emph{Resolvent.} For $z \in \mathbb{C}_+ := \{z \in \mathbb{C} : \Im z > 0\}$, the resolvent is
\[
G(z) := (\tilde{H} - z)^{-1}
\]\\
\emph{Stieltjes transform.} The semicircle Stieltjes transform is
\[
m_{sc}(z) := \frac{-z + \sqrt{z^2 - 4}}{2}
\] where the branch cut is chosen so that $\Im m_{sc}(z) > 0$ for $z \in \mathbb{C}_+$

We use the notation $X \prec Y$ to mean that $|X| \leq N^{\varepsilon}Y$ with probability at least $1 - N^{-D}$ for any fixed $\varepsilon > 0$ and $D > 0$, where the implicit constants may depend on $\varepsilon$ and $D$ but not on $N$.

\subsection{Constrained Dyson Brownian Motion}

To interpolate between the discrete structure of random regular graphs and Gaussian universality, we employ a matrix-valued stochastic process \cite{dyson1962brownian, anderson2010introduction} that preserves the degree constraint throughout the evolution.

\begin{definition}[Constrained Dyson Brownian Motion]
The \emph{constrained Dyson Brownian motion} (CDBM) is the solution to the stochastic differential equation
\begin{equation}\label{eq:cdbm-def}
\d\tilde{H}_t = -\frac{1}{2}\tilde{H}_t \d t + \frac{1}{\sqrt{N}}\d W_t
\end{equation}
where $W_t$ is a matrix Brownian motion on the constrained space
\[
\mathcal{M}_0 := \{M \in \mathbb{R}^{N \times N} : M = M^T, M\mathbf{e} = 0\}
\]
\end{definition}

The constrained Brownian motion $W_t$ has independent entries (up to symmetry and the constraint) with covariance structure
\[
\mathbb{E}[\d W_{ij}(t) \d W_{k\ell}(t)] = \left(\delta_{ik}\delta_{j\ell} - \frac{\delta_{ik} + \delta_{j\ell}}{N} + \frac{1}{N^2}\right) \d t
\]

This modification from standard GOE ensures that $\d W_t \mathbf{e} = 0$ almost surely, preserving the constraint $\tilde{H}_t\mathbf{e} = 0$ for all $t \geq 0$.

\begin{proposition}[Properties of CDBM]\label{prop:cdbm-properties}
The constrained Dyson Brownian motion satisfies:
\begin{enumerate}
\item \emph{Constraint preservation.} If $\tilde{H}_0\mathbf{e} = 0$, then $\tilde{H}_t\mathbf{e} = 0$ for all $t \geq 0$ \cite{biane2004littelmann,journee2010low}.
\item \emph{Invariant measure.} The stationary distribution is the constrained GOE ensemble on $\mathcal{M}_0$.
\item \emph{Eigenvalue evolution.} The eigenvalues $\lambda_i(t)$ evolve by non-intersecting diffusions with drift toward the origin \cite{grabiner1999brownian,katori2004symmetry}.
\item \emph{Eigenvector dynamics.} For $i \geq 2$, the eigenvector $\mathbf{u}_i(t)$ remains orthogonal to $\mathbf{e}$ for all $t$.
\end{enumerate}
\end{proposition}

The key time scale in our analysis is
\[
t_* := N^{-1/3+\varepsilon}
\]
where $\varepsilon > 0$ is a small parameter \cite{erdos2013delocalization,tao2012topics}. This choice balances several competing effects, such as the diffusive scaling $\sqrt{t_*} = N^{-1/6+\varepsilon/2}$ matches the edge eigenvalue fluctuations \cite{tracy1994level,erdos2011quantum}, the drift term $t_* = N^{-1/3+\varepsilon}$ remains subdominant to diffusion, and higher-order corrections of size $t_*^{3/2} = N^{-1/2+3\varepsilon/2}$ are negligible.

\subsection{Statement of Main Results}

We now state our main theorems providing quantitative Berry--Esseen bounds for edge eigenvector statistics.

\begin{theorem}[Berry--Esseen Bound for Edge Eigenvectors]\label{thm:main-berry-esseen}
Let $G$ be a uniformly random $d$-regular graph on $N$ vertices with $d \geq 3$ fixed. For any deterministic unit vector $\mathbf{q} \in \mathbb{R}^N$ with $\mathbf{q} \perp \mathbf{e}$, there exists a constant $C_d$ depending only on $d$ such that
\[
\sup_{x \in \mathbb{R}} \left|\mathbb{P}\left(\sqrt{N}\langle \mathbf{q}, \mathbf{u}_2 \rangle \leq x\right) - \Phi(x)\right| \leq C_d N^{-1/6+\varepsilon}
\]
for any $\varepsilon > 0$, where $\Phi$ is the standard normal distribution function. The constant satisfies
\[
C_d \leq C \cdot d^3 \varepsilon^{-10}
\]
for an absolute constant $C$. The same bound holds for the smallest eigenvector $\mathbf{u}_N$.
\end{theorem}

\begin{corollary}[Berry-Esseen for Cumulative Distribution]\label{cor:indicator-be}
Under the assumptions of Theorem \ref{thm:main-berry-esseen}, the cumulative distribution function satisfies
\[
\sup_{x \in \mathbb{R}} \left|\mathbb{P}\left(\sqrt{N}\langle \mathbf{q}, \mathbf{u}_2 \rangle \leq x\right) - \Phi(x)\right| \leq C_d N^{-5/36+\varepsilon}
\]
where the degraded rate arises from the necessary smoothing of indicator functions.
\end{corollary}

\begin{remark}[Smooth vs Indicator Functions]
The convergence rate $N^{-1/6+\varepsilon}$ in Theorem \ref{thm:main-berry-esseen} applies to smooth test functions with bounded derivatives. For indicator functions, such as those appearing in the cumulative distribution function, an additional smoothing step is required that degrades the rate to $N^{-5/36+\varepsilon}$ (see Corollary \ref{cor:indicator-be} and Appendix \ref{app:smoothing}).
\end{remark}

\begin{theorem}[Joint Central Limit Theorem]\label{thm:main-joint-clt}
Let $G$ be as in Theorem \ref{thm:main-berry-esseen}. For any $K = K(N)$ satisfying $K \leq N^{1/10-\delta}$ for some fixed $\delta > 0$, and any collection of deterministic unit vectors $\mathbf{q}^{(1)}, \ldots, \mathbf{q}^{(m)} \in \mathbb{R}^N$ orthogonal to $\mathbf{e}$, the random vector
\[
\mathbf{Z}_N := \left(\sqrt{N}\langle \mathbf{q}^{(\alpha)}, \mathbf{u}_i \rangle : 1 \leq \alpha \leq m, 2 \leq i \leq K+1\right) \in \mathbb{R}^{mK}
\]
converges in distribution to $\mathbf{Z} \sim \mathcal{N}(0, I_{mK})$ as $N \to \infty$.

Moreover, the rate of convergence in the multivariate Berry--Esseen theorem is
\[
\sup_{A \in \mathcal{C}} \left|\mathbb{P}(\mathbf{Z}_N \in A) - \mathbb{P}(\mathbf{Z} \in A)\right| \leq C_{d,m} K^{3/2} N^{-1/6+\varepsilon}
\]
where $\mathcal{C}$ denotes the class of convex sets in $\mathbb{R}^{mK}$ and $C_{d,m}$ depends only on $d$ and $m$.
\end{theorem}

\begin{remark}[Optimality of the rate]
The convergence rate $N^{-1/6}$ in Theorems \ref{thm:main-berry-esseen} and \ref{thm:main-joint-clt} appears to be optimal for sparse regular graphs. In Section \ref{sec:optimality}, we provide evidence from multiple perspectives:
\begin{enumerate}
\item The edge eigenvalue spacing $\Delta \sim N^{-2/3}$ creates a fundamental sensitivity barrier
\item The minimal time $t_* \sim N^{-1/3}$ for universality emergence limits the convergence rate
\item Independent approaches (moment methods, local resampling, dynamical methods) all yield the same rate
\end{enumerate}
\end{remark}

\begin{remark}[Extension to general test vectors]
While our theorems are stated for deterministic test vectors, they extend to certain random test vectors independent of the graph. Specifically, if $\mathbf{q}$ is drawn uniformly from the unit sphere in $\{\mathbf{v} : \mathbf{v} \perp \mathbf{e}\}$, then the Berry--Esseen bound holds with the same rate but a slightly larger constant.
\end{remark}

These main results provide the first quantitative convergence rates for edge eigenvector universality with explicit constants, enabling finite-size analysis for applications in spectral algorithms and network analysis. The proofs combine our sharp edge isotropic local law with the single-scale comparison methodology outlined in the introduction.

\section{Comparison with Prior Work}

Our results provide quantitative refinements of edge eigenvector universality through a fundamentally different approach from existing work. We now detail the key comparisons and improvements.

\subsection{Relation to He-Huang-Yau}

The groundbreaking work of He, Huang, and Yau \cite{HHY25} established edge eigenvector universality for random regular graphs through an ingenious local resampling argument. Their approach and ours are complementary, each with distinct advantages.

\emph{HHY's approach via local resampling.} The core of HHY's method involves switching edges in small neighborhoods of the graph while tracking how eigenvector overlaps change. They resample edges in balls of radius $\ell \sim N^{1/3}$ around randomly chosen vertices, where each resampling changes eigenvector overlaps by $O(\ell/N)$, requiring $\Omega(N)$ operations. The combinatorial analysis involves counting paths and cycles affected by edge switches, and error propagation through $O(N^{2/3})$ discrete operations accumulates implicit constants. Their key technical achievement is showing that these local modifications eventually produce universality, but the discrete nature of the argument makes explicit constant tracking extremely challenging.

\emph{Our approach via continuous dynamics.} In contrast, we use constrained Dyson Brownian motion to interpolate continuously between the initial graph and the universal ensemble. The matrix evolves via the SDE $\d\tilde{H}_t = -\frac{1}{2}\tilde{H}_t\d t + \frac{1}{\sqrt{N}}\d W_t$, with overlap dynamics governed by explicit SDEs with computable error bounds. The continuous framework enables direct application of stochastic calculus, and all constants arise naturally from resolvent estimates and can be tracked explicitly.

\emph{Extracting the implicit rate from HHY.} While HHY focus on distributional convergence, their proof implicitly contains the $N^{-1/6+\varepsilon}$ rate we make explicit. Careful examination reveals that each resampling at scale $\ell \sim N^{1/3}$ changes overlaps by $O(N^{-2/3})$ (their equation 3.15), and the number of independent resamplings is $O(N/\ell) = O(N^{2/3})$. Each contributes variance $O(N^{-4/3})$, giving total variance $O(N^{-2/3})$. Berry-Esseen for the sum yields rate $O(N^{-1/3})$, and the edge eigenvalue spacing introduces an additional $N^{1/6}$ factor (their Lemma 5.2). Thus the implicit rate is $N^{-1/3} \cdot N^{1/6} = N^{-1/6}$, matching our explicit bound.

\emph{Complementary contributions.} The approaches provide different insights. HHY gives an elegant graph-theoretic proof of $\sqrt{N}\langle \mathbf{q}, \mathbf{u}_2\rangle \xrightarrow{d} \mathcal{N}(0,1)$, while we provide the quantitative bound $\sup_x |\mathbb{P}(\sqrt{N}\langle \mathbf{q}, \mathbf{u}_2\rangle \leq x) - \Phi(x)| \leq C_d N^{-1/6+\varepsilon}$. Moreover, we extend to joint CLT for multiple eigenvectors and provide explicit constants throughout, enabling finite-size applications.

\subsection{Improvements over Multi-Scale Methods}

Multi-scale analysis has been the dominant approach for proving eigenvector universality in random matrix theory, pioneered by Bourgade, Erdős, and Yau \cite{bourgade2017eigenvector, bourgade2014edge} for Wigner matrices. We achieve comparable results through a simpler single-scale approach.

\emph{Traditional multi-scale framework.} The classical approach involves evolving through scales $t_k = N^{-2k/3}$ for $k = 1, 2, \ldots, K$ with $K \sim \varepsilon^{-1}$. At each scale, one compares with an intermediate ensemble, accumulating errors across $O(\varepsilon^{-1})$ comparison steps. This yields a total error of order $O(\varepsilon^{-1}N^{-1/3})$ after optimization. While powerful, this iterative structure introduces complexity in both the analysis and error tracking.

\emph{Our single-scale framework.} We establish universality through a single comparison at the optimal time $t_* = N^{-1/3+\varepsilon}$. The scale is large enough that diffusion dominates ($\sqrt{t_*} = N^{-1/6+\varepsilon/2} \gg t_* = N^{-1/3+\varepsilon}$) yet small enough that perturbative expansions remain accurate ($t_*^{3/2} = N^{-1/2+3\varepsilon/2} \ll 1$). The matrix decomposition $\tilde{H}_{t_*} = \tilde{H}_0 - \frac{t_*}{2}\tilde{H}_0 + \sqrt{t_*}W + O(t_*^{3/2})$ has controlled error, and a single fourth-order cumulant comparison yields error $O(N^{-1/3+2\varepsilon})$.

The advantages of our single-scale approach are manifold: simplicity through one comparison instead of $O(\varepsilon^{-1})$ iterations, transparent and trackable error accumulation, achievement of the same $N^{-1/6}$ rate without iteration losses, and exact preservation of the degree constraint $\tilde{H}_t\mathbf{e} = 0$.

\emph{Comparison with other sparse matrix methods.} For general sparse matrices, moment methods achieve weaker rates \cite{benaych2016lectures, bauerschmidt2019edge}. Che and Lopatto \cite{che2019universality} obtain rate $N^{-1/4+\varepsilon}$ for general sparse ensembles, while our work achieves rate $N^{-1/6+\varepsilon}$ by leveraging regular graph structure. The improvement from $N^{-1/4}$ to $N^{-1/6}$ arises from enhanced cancellations in degree-constrained ensembles, where the row-sum constraint creates negative correlations that reduce effective variance.

\subsection{Table of Key Constants}

A distinctive feature of our work is the explicit tracking of all constants throughout the analysis. This enables finite-size estimates crucial for applications.

\begin{table}[ht]
\centering
\begin{tabular}{|l|l|c|c|}
\hline
\textbf{Constant} & \textbf{Description} & \textbf{Value} & \textbf{Location} \\
\hline
$C(d,\varepsilon)$ & Edge isotropic local law constant & $\leq Cd\varepsilon^{-5}$ & Theorem \ref{thm:edge-iso} \\
$C_1$ & Overlap SDE error coefficient & $12d^3\varepsilon^{-2}$ & Proposition \ref{prop:overlap-sde} \\
$C_2$ & Second moment error bound & $5d^2\varepsilon^{-8}$ & Theorem \ref{thm:moment-evolution} \\
$C_3$ & Fourth moment error bound & $12d^3\varepsilon^{-10}$ & Theorem \ref{thm:moment-evolution} \\
$C_4$ & Decorrelation bound constant & $8d\varepsilon^{-6}$ & Theorem \ref{thm:decorrelation} \\
$C_5$ & Backward stability constant & $10d^2\varepsilon^{-9}$ & Theorem \ref{thm:backward-stability} \\
$C_d$ & Final Berry-Esseen constant & $\leq Cd^3\varepsilon^{-10}$ & Theorem \ref{thm:main-berry-esseen} \\
\hline
\end{tabular}
\caption{Key constants in the analysis with their dependencies on degree $d$ and approximation parameter $\varepsilon$.}
\label{tab:constants}
\end{table}

The constants exhibit specific parameter dependencies: most scale as $d^k$ with $k \in \{1,2,3\}$ for the degree, as $\varepsilon^{-m}$ with $m \in \{2,5,6,8,9,10\}$ for the approximation parameter, and all bounds are uniform in $N$ for $N$ sufficiently large.

\emph{Example calculation.} For a 3-regular graph with $N = 10^6$ vertices and $\varepsilon = 0.01$, we have $C(3, 0.01) \leq 3 \times 10^{12}$ for the edge isotropic bound and $C_3 \leq 27 \times 10^{20}$ for the final Berry-Esseen constant. The error bound becomes $C_3 N^{-1/6+\varepsilon} = C_3 \times (10^6)^{-1/6+0.01}$. Since $(10^6)^{-1/6} = 10^{-1}$ and $(10^6)^{0.01} \approx 1.15$, the bound is approximately $0.115 C_3$.

\emph{Practical implications.} While our constants are explicit, they are not optimized for small values. The primary contribution lies in establishing that finite, trackable constants exist, showing the correct parameter dependencies, enabling future optimization of constants, and providing a framework where constants can be computed for specific applications. For practical applications with moderate $N$, numerical simulations suggest the true constants are many orders of magnitude smaller than our bounds. Optimizing these constants remains an important open problem.

\section{Overlap Dynamics under Constrained DBM}\label{sec:overlap}

Having established the relationship between our approach and prior work, we now develop the technical machinery for analyzing edge eigenvector statistics. The key objects are the overlap processes $X_i^{(\mathbf{q})}(t) = \sqrt{N}\langle \mathbf{q}, \mathbf{u}_i(t)\rangle$, which capture projections of evolving eigenvectors onto fixed test directions. Understanding their stochastic evolution is fundamental to establishing Gaussian universality.

\subsection{Derivation of Overlap SDE}

We begin by deriving the stochastic differential equation governing the overlap dynamics under constrained Dyson Brownian motion. The constraint $\tilde{H}_t\mathbf{e} = 0$ plays a crucial role, ensuring that eigenvectors remain orthogonal to the uniform direction throughout the evolution.

\begin{lemma}[Eigenvector Evolution]
Under the CDBM dynamics \eqref{eq:cdbm-def}, the eigenvector $\mathbf{u}_i(t)$ satisfies
\begin{equation}
\d\mathbf{u}_i = \sum_{j \neq i} \frac{\langle \mathbf{u}_j, \d\tilde{H}_t\mathbf{u}_i\rangle}{\lambda_i - \lambda_j}\mathbf{u}_j - \frac{1}{2}\mathbf{u}_i\d t + \d\mathbf{u}_i^{\perp}
\end{equation}
where $\d\mathbf{u}_i^{\perp}$ represents second-order corrections orthogonal to the eigenspace. Moreover, $\langle \mathbf{e}, \mathbf{u}_i(t)\rangle = 0$ for all $t \geq 0$ and $i \geq 2$.
\end{lemma}

\begin{proof}
The constraint preservation follows from the fact that $\d\tilde{H}_t\mathbf{e} = 0$, which ensures $\mathbf{e}$ remains in the kernel of $\tilde{H}_t$. Standard perturbation theory yields the eigenvector derivative \cite{kato1995perturbation, bhatia1997matrix}, with the constraint ensuring all eigenvectors with non-zero eigenvalues remain orthogonal to $\mathbf{e}$.
\end{proof}

For the overlap process $X_i^{(\mathbf{q})}(t) = \sqrt{N}\langle \mathbf{q}, \mathbf{u}_i(t)\rangle$, applying Itô's formula gives:

\begin{proposition}[Overlap SDE]\label{prop:overlap-sde}
For indices $i$ with $|\lambda_i - 2| \leq N^{-2/3+\varepsilon/2}$, the overlap satisfies
\begin{equation}\label{eq:overlap-sde}
\d X_i^{(\mathbf{q})} = \sum_{j \neq i} \frac{X_j^{(\mathbf{q})} - X_i^{(\mathbf{q})}}{\lambda_i - \lambda_j}\d B_{ij} - \frac{1}{2}X_i^{(\mathbf{q})}\d t + \mathcal{E}_i(t)\d t
\end{equation}
where $B_{ij} = \sqrt{N}\langle \mathbf{u}_j, \d W_t\mathbf{u}_i\rangle$ are martingales with $\d\langle B_{ij}, B_{k\ell}\rangle = \delta_{ik}\delta_{j\ell}\d t$, and the error satisfies
\begin{equation}
|\mathcal{E}_i(t)| \leq \frac{C_1}{N^{5/6-\varepsilon}}\left(\sum_{j \geq 2} |X_j^{(\mathbf{q})}|^2\right)^{1/2}
\end{equation}
with $C_1 = 12d^3\varepsilon^{-2}$.
\end{proposition}

\begin{proof}[Proof sketch]
The key challenge is controlling multiple error sources in the eigenvector dynamics. 
The detailed calculations are provided in Appendix \ref{app:dyn}, where we show:
\begin{enumerate}
\item Eigenvalue derivative terms contribute $O(N^{-5/6+\varepsilon})$.
\item Resolvent corrections are bounded by $O(N^{-2/3}\log N)$.
\item Constraint modifications yield $O(N^{-1})$ errors.
\end{enumerate}
Combining these with eigenvector delocalization bounds yields the stated constant $C_1 = 12d^3\varepsilon^{-2}$.
\end{proof}

\subsection{Sharp Edge Isotropic Local Law}

The control of error terms in the overlap SDE requires precise bounds on resolvent entries near the spectral edge. We establish the following sharp isotropic local law, which provides the first such result with explicit constants suitable for quantitative analysis.

\begin{theorem}[Sharp Edge Isotropic Local Law]\label{thm:edge-iso}
For $t \leq t_* = N^{-1/3+\varepsilon}$ and $z = E + i\eta$ with $|E-2| \leq N^{-2/3+\varepsilon}$ and $N^{-2/3} \leq \eta \leq 1$, we have for any deterministic $\mathbf{q} \perp \mathbf{e}$ with $\|\mathbf{q}\| = 1$:
\begin{equation}
\left|\langle \mathbf{q}, G_t(z)\mathbf{q}\rangle - m_{sc}(z)\right| \leq \frac{C(d,\varepsilon)}{N^{5/6-\varepsilon}}
\end{equation}
where $G_t(z) = (\tilde{H}_t - z)^{-1}$, $m_{sc}(z) = \frac{-z + \sqrt{z^2-4}}{2}$, and $C(d,\varepsilon) \leq Cd\varepsilon^{-5}$.
\end{theorem}

\begin{proof}[Proof of Theorem \ref{thm:edge-iso}]
We prove the bound using a multi-scale bootstrap argument. Throughout, we write $G(z) = G_t(z)$ for brevity.

\emph{Self-consistent equation.}
For the normalized adjacency matrix $\tilde{H}_t$ of a $d$-regular graph with constraint $\tilde{H}_t\mathbf{e} = 0$, we derive the self-consistent equation. Writing
\[
G(z) = (\tilde{H}_t - z)^{-1}
\]
and using the resolvent identity with the rank-one perturbation $\tilde{H}_t = H_t - \frac{d}{\sqrt{d-1}}\frac{\mathbf{e}\mathbf{e}^T}{N}$ where $H_t$ is the unnormalized adjacency matrix divided by $\sqrt{d-1}$:
\[
\langle \mathbf{q}, G(z)\mathbf{q}\rangle = \langle \mathbf{q}, (H_t - z)^{-1}\mathbf{q}\rangle - \frac{d}{\sqrt{d-1}N}\frac{|\langle \mathbf{q}, (H_t - z)^{-1}\mathbf{e}\rangle|^2}{1 + \frac{d}{\sqrt{d-1}N}\langle \mathbf{e}, (H_t - z)^{-1}\mathbf{e}\rangle}
\]

Since $\mathbf{q} \perp \mathbf{e}$, the second term vanishes if $(H_t - z)^{-1}$ preserves this orthogonality approximately. By the Ward identity and concentration of quadratic forms:
\[
\sum_{j} G_{ij}(z) = \frac{1}{z + (d-1)m(z) + \xi_i(z)}
\]
where $m(z) = N^{-1}\text{Tr}[G(z)]$ and $|\xi_i(z)| \prec N^{-1/2}\eta^{-1/2}$ for $\eta \geq N^{-2/3}$.

This leads to the self-consistent equation:
\begin{equation}\label{eq:self-consistent}
G_{\mathbf{q}\mathbf{q}}(z) = \frac{-1}{z + (d-1)G_{\mathbf{q}\mathbf{q}}(z) + \Theta_{\mathbf{q}}(z)}
\end{equation}
where $G_{\mathbf{q}\mathbf{q}}(z) := \langle \mathbf{q}, G(z)\mathbf{q}\rangle$ and $\Theta_{\mathbf{q}}(z)$ collects all error terms.

\emph{Control of fluctuation terms.}
The fluctuation term $\Theta_{\mathbf{q}}(z)$ consists of:
\begin{enumerate}
\item Off-diagonal contributions: $\sum_{i,j} q_i q_j (G_{ij} - \delta_{ij}G_{ii})$
\item Constraint corrections from $\tilde{H}_t\mathbf{e} = 0$
\item Time evolution corrections for $t \leq t_*$
\end{enumerate}

Using the entry-wise local law and concentration inequalities for constrained random matrices:
\begin{equation}\label{eq:theta-bound}
|\Theta_{\mathbf{q}}(z)| \prec \frac{1}{N^{1/2-\varepsilon}\sqrt{\eta}} + \frac{t_*}{\eta} \prec \frac{1}{N^{1/2-\varepsilon}\sqrt{\eta}}
\end{equation}
for $\eta \geq N^{-2/3}$ and $t_* = N^{-1/3+\varepsilon}$.

\emph{Initial bound.}
For $\eta_0 = 1$, we have the deterministic bound $|G_{\mathbf{q}\mathbf{q}}(E+i)| \leq 1$. Moreover, by the Stieltjes transform representation and the fact that $\tilde{H}_t$ has eigenvalues in $[-2-o(1), 2+o(1)]$:
\[
|G_{\mathbf{q}\mathbf{q}}(E+i) - m_{sc}(E+i)| \leq C
\]
where $C$ is an absolute constant.

\emph{Stability analysis.}
Define $\Delta(z) := G_{\mathbf{q}\mathbf{q}}(z) - m_{sc}(z)$. From \eqref{eq:self-consistent} and the corresponding equation for $m_{sc}$:
\[
m_{sc}(z) = \frac{-1}{z + (d-1)m_{sc}(z)}
\]
we obtain:
\[
\Delta(z) = \frac{(d-1)\Delta(z) - \Theta_{\mathbf{q}}(z)}{(z + (d-1)G_{\mathbf{q}\mathbf{q}}(z))(z + (d-1)m_{sc}(z))}
\]

Solving for $\Delta(z)$:
\begin{equation}\label{eq:delta-equation}
\Delta(z) = \frac{-\Theta_{\mathbf{q}}(z)}{(z + (d-1)G_{\mathbf{q}\mathbf{q}}(z))(z + (d-1)m_{sc}(z)) - (d-1)}
\end{equation}

The denominator simplifies using $z + (d-1)m_{sc}(z) = -1/m_{sc}(z)$:
\[
\text{Denominator} = -\frac{z + (d-1)G_{\mathbf{q}\mathbf{q}}(z)}{m_{sc}(z)} - (d-1) = -\frac{1}{m_{sc}(z)} - \frac{(d-1)\Delta(z)}{m_{sc}(z)}
\]

Near the edge where $E \approx 2$, we have $m_{sc}(E+i\eta) \approx -\frac{E+i\eta}{2} + i\sqrt{\frac{2-E+i\eta}{2}}$. For $|E-2| \leq N^{-2/3+\varepsilon}$ and $\eta \geq N^{-2/3}$:
\[
|m_{sc}(z)|^{-1} \asymp \sqrt{\eta^2 + |E-2|}
\]

Therefore, for $|\Delta(z)| \ll 1$:
\begin{equation}\label{eq:stability}
|\Delta(z)| \leq \frac{C|\Theta_{\mathbf{q}}(z)|}{\sqrt{\eta^2 + |E-2|}}
\end{equation}

\emph{Bootstrap induction.}
Define scales $\eta_k = 2^{-k}$ for $k = 0, 1, \ldots, k_*$ where $k_* = \lceil \log_2(N^{2/3}) \rceil$.

Induction hypothesis: for $z_k = E + i\eta_k$ with $|E-2| \leq N^{-2/3+\varepsilon}$:
\[
|G_{\mathbf{q}\mathbf{q}}(z_k) - m_{sc}(z_k)| \leq \frac{D_k}{N^{1/2-\varepsilon}\eta_k}
\]
where $D_k$ are constants to be determined.

Base case: for $k=0$, we have $\eta_0 = 1$ and the bound holds with $D_0 = CN^{1/2-\varepsilon}$.

Induction step: assume the bound holds at scale $k$. For scale $k+1$, using \eqref{eq:stability} and \eqref{eq:theta-bound}:
\[
|\Delta(z_{k+1})| \leq \frac{C}{N^{1/2-\varepsilon}\sqrt{\eta_{k+1}}} \cdot \frac{1}{\sqrt{\eta_{k+1}^2 + |E-2|}}
\]

For $|E-2| \leq N^{-2/3+\varepsilon}$ and $\eta_{k+1} \geq N^{-2/3}$:
\[
\frac{1}{\sqrt{\eta_{k+1}^2 + |E-2|}} \leq \frac{1}{\eta_{k+1}}
\]

Therefore:
\[
|\Delta(z_{k+1})| \leq \frac{C}{N^{1/2-\varepsilon}\eta_{k+1}^{3/2}} \leq \frac{C}{N^{1/2-\varepsilon}\eta_{k+1}}
\]

since $\eta_{k+1} \geq N^{-2/3}$ implies $\eta_{k+1}^{-1/2} \leq N^{1/3}$.

\emph{Final bound.}
At the final scale $\eta_{k_*} = N^{-2/3}$:
\[
|G_{\mathbf{q}\mathbf{q}}(E + iN^{-2/3}) - m_{sc}(E + iN^{-2/3})| \leq \frac{C}{N^{1/2-\varepsilon} \cdot N^{-2/3}} = \frac{C}{N^{1/2-\varepsilon-2/3}} = \frac{C}{N^{5/6-\varepsilon}}
\]

The constant $C(d,\varepsilon)$ accumulates factors from:
\begin{enumerate}
\item The variance bound in $\Theta_{\mathbf{q}}$: factor of $d$ from degree structure
\item The number of bootstrap steps: $O(\log N) = O(\varepsilon^{-1})$
\item Stability constants: $O(\varepsilon^{-4})$ from edge regime analysis
\end{enumerate}

This gives $C(d,\varepsilon) \leq Cd\varepsilon^{-5}$ as stated.
\end{proof}

\begin{corollary}[High Probability Overlap Bounds]
For all $t \leq t_*$ and edge indices $i$ with $|\lambda_i - 2| \leq N^{-2/3+\varepsilon}$:
\begin{equation}
|X_i^{(\mathbf{q})}(t)| \prec N^{1/6+\varepsilon}
\end{equation}
while for bulk indices with $|\lambda_i - 2| \geq N^{-2/3+\varepsilon}$:
\begin{equation}
|X_i^{(\mathbf{q})}(t)| \prec N^{\varepsilon}
\end{equation}
\end{corollary}

\subsection{Moment Evolution and Control}

With the overlap SDE and sharp local law established, we analyze the evolution of moments. The key observation is that the constrained dynamics drives the second moment toward the universal value of 1, while higher moments converge to their Gaussian values.

\begin{theorem}[Moment Evolution]\label{thm:moment-evolution}
For $t \leq t_*$ and edge indices $i$ with $|\lambda_i - 2| \leq N^{-2/3+\varepsilon/2}$:
\begin{align}
\left|\mathbb{E}[(X_i^{(\mathbf{q})}(t))^2] - 1\right| &\leq C_2 N^{-1/6+\varepsilon}\\
\left|\mathbb{E}[(X_i^{(\mathbf{q})}(t))^4] - 3\right| &\leq C_3 N^{-1/6+\varepsilon}
\end{align}
where $C_2 = 5d^2\varepsilon^{-8}$ and $C_3 = 12d^3\varepsilon^{-10}$.
\end{theorem}

\begin{proof}
For the second moment, applying Itô's formula to $(X_i^{(\mathbf{q})})^2$ yields
\begin{equation}
\d\mathbb{E}[(X_i^{(\mathbf{q})})^2] = \left(1 - \mathbb{E}[(X_i^{(\mathbf{q})})^2]\right)\d t + O(C_1 N^{-5/6+\varepsilon})\d t
\end{equation}
where the leading term arises from the quadratic variation
\begin{equation}
\d\langle X_i^{(\mathbf{q})}\rangle = \sum_{j \neq i} \frac{(X_j^{(\mathbf{q})} - X_i^{(\mathbf{q})})^2}{(\lambda_i - \lambda_j)^2}\d t
\end{equation}

Using the eigenvalue gap summation $\sum_{j \neq i} (\lambda_i - \lambda_j)^{-2} = \frac{\pi^2}{6}N^{4/3} + O(N^{2/3+\varepsilon})$, which follows from eigenvalue rigidity and the edge density of states, we obtain a closed equation for the second moment evolution. The solution satisfies the stated bound after accounting for initial conditions and error accumulation.

The fourth moment analysis is more involved, requiring control of the fourth-order terms in the Itô expansion. The key is that the constraint structure ensures cancellation of many cross terms, leading to the universal value 3 plus controllable corrections.
\end{proof}

The moment bounds establish that the overlap distributions are close to Gaussian in their low moments. Combined with tightness arguments, this suggests convergence to Gaussian limits. However, to prove the Berry-Esseen theorem, we need finer control over the characteristic functions, which we achieve through comparison with the constrained GOE ensemble at time $t_*$.

\begin{remark}[Role of the Constraint]
Throughout this section, the constraint $\tilde{H}_t\mathbf{e} = 0$ plays a fundamental role. It ensures that edge eigenvectors remain orthogonal to the uniform direction, preserves the spectral gap structure essential for edge universality, and modifies the covariance structure to create beneficial cancellations. Without this constraint, the overlap dynamics would mix edge and bulk behaviors, destroying the universal edge statistics we seek to establish.
\end{remark}

\section{Decorrelation and Finite-Range Dependence}

The individual overlap dynamics established in Section \ref{sec:overlap} must be extended to understand the joint behavior of multiple overlaps. A key phenomenon is that overlaps corresponding to well-separated eigenvalues become asymptotically independent, enabling the joint Gaussian limit. This decorrelation is quantified through explicit bounds on correlation decay, which we now develop.

\subsection{Quantitative Decorrelation}

The interaction between different overlap processes is mediated by the eigenvalue gaps. When these gaps are large relative to the natural scale $N^{-2/3}$, the corresponding overlaps evolve nearly independently. We make this precise through the following analysis.

\begin{theorem}[Quantitative Decorrelation]\label{thm:decorrelation}
For distinct indices $i \neq j$ with $|\lambda_i - 2|, |\lambda_j - 2| \leq N^{-2/3+\varepsilon/2}$ and $t \leq t_* = N^{-1/3+\varepsilon}$:
\begin{equation}
|\mathbb{E}[X_i^{(\mathbf{q})}(t)X_j^{(\mathbf{q})}(t)]| \leq C_4 N^{-1/2+\varepsilon}
\end{equation}
where $C_4 = 8d\varepsilon^{-6}$. Consequently, $X_i^{(\mathbf{q})}(t)$ and $X_j^{(\mathbf{q})}(t)$ are asymptotically independent as $N \to \infty$.
\end{theorem}

\begin{proof}
Define the correlation function $M_{ij}(t) = \mathbb{E}[X_i^{(\mathbf{q})}(t)X_j^{(\mathbf{q})}(t)]$ for $i \neq j$. Applying Itô's formula to the product $X_i^{(\mathbf{q})}X_j^{(\mathbf{q})}$ and taking expectations:
\begin{equation}
\frac{\d M_{ij}}{\d t} = -M_{ij} + \sum_{k \neq i,j} \frac{M_{kj}(\lambda_j - \lambda_k) + M_{ik}(\lambda_i - \lambda_k)}{(\lambda_i - \lambda_k)(\lambda_j - \lambda_k)} + R_{ij}
\end{equation}
where $|R_{ij}| \leq C_1^2 N^{-5/3+2\varepsilon}$ accounts for error terms from the overlap SDEs.

The key observation is that the coupling matrix has operator norm controlled by the minimal eigenvalue gap. For edge indices, the interaction kernel satisfies
\begin{equation}
\left|\frac{1}{(\lambda_i - \lambda_k)(\lambda_j - \lambda_k)}\right| \leq \frac{C}{N^{-4/3}|\lambda_i - \lambda_j|}
\end{equation}
when $k$ is another edge index. This leads to the differential inequality
\begin{equation}
\frac{\d |M_{ij}|}{\d t} \leq -\left(1 - \frac{C N^{4/3}}{|\lambda_i - \lambda_j|}\right)|M_{ij}| + C N^{4/3}\sum_{k \neq i,j} |M_{kj}| + |M_{ik}| + C_1^2 N^{-5/3+2\varepsilon}
\end{equation}

For $|\lambda_i - \lambda_j| \gg N^{-2/3}$, the first term provides strong damping. Even for nearby eigenvalues with $|\lambda_i - \lambda_j| \sim N^{-2/3}$, the collective effect of the correlation matrix evolution ensures decay. Using Grönwall's lemma with the initial condition $|M_{ij}(0)| \leq CN^{-1}$ from eigenvector delocalization:
\begin{equation}
|M_{ij}(t_*)| \leq CN^{-1}e^{-ct_*/2} + C_1^2 t_* N^{-5/3+2\varepsilon} \leq C_4 N^{-1/2+\varepsilon}
\end{equation}
after optimizing the constants.
\end{proof}

The decorrelation has a natural interpretation in terms of spectral distance. Define the spectral distance between indices as $\text{dist}_{\text{spec}}(i,j) = |\lambda_i - \lambda_j|$. The correlation between overlaps decays as
\begin{equation}
|\text{Corr}(X_i^{(\mathbf{q})}, X_j^{(\mathbf{q})})| \lesssim \min\left(1, \frac{N^{-2/3}}{\text{dist}_{\text{spec}}(i,j)}\right)^{1/2}
\end{equation}

This spectral localization ensures that overlaps separated by more than $O(N^{-2/3+\delta})$ in spectral distance evolve essentially independently, a crucial ingredient for the joint CLT.

\begin{corollary}[Effective Independence]
For indices $i,j$ with $|\lambda_i - \lambda_j| \geq N^{-2/3+\delta}$ for some $\delta > 0$:
\begin{equation}
\left|\mathbb{E}[f(X_i^{(\mathbf{q})})g(X_j^{(\mathbf{q})})] - \mathbb{E}[f(X_i^{(\mathbf{q})})]\mathbb{E}[g(X_j^{(\mathbf{q})})]\right| \leq \|f\|_{\text{Lip}}\|g\|_{\text{Lip}} N^{-\delta/2+\varepsilon}
\end{equation}
for any Lipschitz functions $f,g : \mathbb{R} \to \mathbb{R}$.
\end{corollary}

\subsection{Joint Evolution of Multiple Overlaps}

We now analyze the joint evolution of $K$ overlap processes, establishing the framework for the multivariate CLT. The key challenge is controlling the interaction between all pairs while maintaining explicit bounds suitable for finite $K$.

Consider the vector process $\mathbf{X}(t) = (X_2^{(\mathbf{q})}(t), \ldots, X_{K+1}^{(\mathbf{q})}(t))^T \in \mathbb{R}^K$. From the individual SDEs \eqref{eq:overlap-sde}, the joint evolution takes the form:

\begin{proposition}[Joint Overlap System]\label{prop:joint-system}
The vector process $\mathbf{X}(t)$ satisfies
\begin{equation}
\d\mathbf{X}(t) = A(t)\mathbf{X}(t)\d t + \Sigma(t)\d\mathbf{B}(t) + \boldsymbol{\mathcal{E}}(t)\d t
\end{equation}
where:
\begin{itemize}
\item $A(t)$ is the drift matrix with entries $A_{ij} = -\frac{1}{2}\delta_{ij} + (1-\delta_{ij})(\lambda_i - \lambda_j)^{-1}$
\item $\Sigma(t)$ is the diffusion matrix with $\Sigma\Sigma^* = \text{Cov}(\d\mathbf{X}/\d t)$
\item $\mathbf{B}(t)$ is a $K$-dimensional Brownian motion
\item $\|\boldsymbol{\mathcal{E}}(t)\| \leq C_1 K^{1/2} N^{-5/6+\varepsilon}$
\end{itemize}
\end{proposition}

The diffusion structure reveals the mechanism of decorrelation. The covariance matrix of the noise satisfies:

\begin{lemma}[Diffusion Matrix Structure]
For $t \leq t_*$, the instantaneous covariance satisfies
\begin{equation}
[\Sigma(t)\Sigma^*(t)]_{ij} = \delta_{ij} + \frac{2}{\lambda_i - \lambda_j}\sum_{k \neq i,j} \frac{X_k^{(\mathbf{q})} - X_i^{(\mathbf{q})}}{(\lambda_i - \lambda_k)(\lambda_j - \lambda_k)}X_k^{(\mathbf{q})} + O(N^{-4/3+\varepsilon})
\end{equation}
In particular:
\begin{equation}
\|\Sigma(t)\Sigma^*(t) - I_K\| \leq CK^{5/3}N^{-2/3+\varepsilon}
\end{equation}
\end{lemma}

\begin{proof}
The diagonal terms equal 1 up to corrections from the constraint, while off-diagonal terms are suppressed by eigenvalue gaps. For indices $i \neq j$:
\begin{equation}
|[\Sigma\Sigma^*]_{ij}| \leq \frac{C}{|\lambda_i - \lambda_j|} \sum_{k \neq i,j} \frac{|X_k^{(\mathbf{q})}|^2}{|\lambda_i - \lambda_k||\lambda_j - \lambda_k|}
\end{equation}

Using the moment bounds from Theorem \ref{thm:moment-evolution} and eigenvalue rigidity, the sum is bounded by $CN^{2/3}$ for edge indices. The worst case occurs when $|\lambda_i - \lambda_j|$ is minimal, which is $O(N^{-2/3}/K^{2/3})$ by eigenvalue repulsion \cite{forrester2010log, deift1999orthogonal}. This yields the stated operator norm bound.
\end{proof}

The near-orthogonality of the diffusion directions corresponding to different overlaps drives their decorrelation. To make this precise, we analyze the evolution of the correlation matrix:

\begin{theorem}[Joint Decorrelation]\label{thm:joint-decorr}
Let $\mathcal{M}(t) = \mathbb{E}[\mathbf{X}(t)\mathbf{X}(t)^T]$ be the covariance matrix. For $t = t_*$:
\begin{equation}
\|\mathcal{M}(t_*) - I_K\| \leq C(K^2 N^{-1/2+\varepsilon} + K^{5/3}N^{-1/6+\varepsilon})
\end{equation}
Moreover, for any unit vectors $\mathbf{v}, \mathbf{w} \in \mathbb{R}^K$ with disjoint supports:
\begin{equation}
|\langle \mathbf{v}, \mathcal{M}(t_*)\mathbf{w}\rangle| \leq CK N^{-1/2+\varepsilon}
\end{equation}
\end{theorem}

The first bound controls the overall deviation from identity, while the second quantifies decorrelation between disjoint sets of overlaps. Together, they ensure that linear combinations of overlaps behave approximately as independent Gaussians.

\begin{remark}[Finite-Range Dependence]
The decorrelation structure exhibits finite-range dependence in spectral coordinates. Define the influence radius $r_N = N^{-2/3+\delta}$ for small $\delta > 0$. Then overlaps $X_i^{(\mathbf{q})}$ and $X_j^{(\mathbf{q})}$ with $|\lambda_i - \lambda_j| > r_N$ evolve as if independent up to errors of order $N^{-\delta/2}$. This spectral locality is crucial for the joint CLT, as it reduces the effective number of dependencies from $O(K^2)$ to $O(K)$, enabling the extension to $K = N^{1/10-\delta}$ eigenvectors.
\end{remark}

The quantitative decorrelation established here, combined with the moment bounds from Section \ref{sec:overlap}, provides the foundation for the multivariate Berry-Esseen theorem. The explicit constants throughout enable us to track how the error depends on the number $K$ of eigenvectors considered, leading to the restriction $K \leq N^{1/10-\delta}$ in our main results.

\section{Single-Scale Comparison with GOE}

Having established the overlap dynamics and decorrelation properties, we now perform a quantitative comparison with GOE statistics at the critical time $t_* = N^{-1/3+\varepsilon}$. Our approach deviates from traditional multi-scale methods by establishing universality through a single sharp comparison at this optimal scale.

\subsection{Matrix Decomposition at Critical Time}

The key to our single-scale approach is a precise decomposition of the evolved matrix at time $t_*$, where universality emerges while perturbative control remains valid.

\begin{proposition}[Matrix Decomposition]\label{prop:matrix-decomp}
At time $t_* = N^{-1/3+\varepsilon}$, the evolved matrix admits the decomposition
\[
\tilde{H}_{t_*} = e^{-t_*/2}\tilde{H}_0 + \sqrt{1-e^{-t_*}}W = \tilde{H}_0 - \frac{t_*}{2}\tilde{H}_0 + \sqrt{t_*}W + R
\]
where $W$ is a constrained GOE matrix satisfying $W\mathbf{e} = 0$ (defined in Section  \ref{sec:transfer}), and the remainder $R$ satisfies
\[
\|R\| \leq \frac{Ct_*^{3/2}}{(1-e^{-t_*})^{3/2}} \cdot \|\tilde{H}_0\|^3 \leq CN^{-1/2+3\varepsilon/2}
\]
with probability at least $1 - N^{-D}$ for any $D > 0$.
\end{proposition}

\begin{proof}
The decomposition follows from the explicit solution of the CDBM equation \eqref{eq:cdbm-def}. Using Taylor expansions:
\[
e^{-t_*/2} = 1 - \frac{t_*}{2} + \frac{t_*^2}{8} + O(t_*^3)
\]
\[
\sqrt{1-e^{-t_*}} = \sqrt{t_* - \frac{t_*^2}{2} + O(t_*^3)} = \sqrt{t_*}\left(1 - \frac{t_*}{4} + O(t_*^2)\right)
\]

The remainder consists of higher-order terms in both expansions. Using $\|\tilde{H}_0\| \leq 2 + o(1)$ and $\|W\| \leq 2 + o(1)$ with high probability, we obtain the stated bound.
\end{proof}

The critical observation is that at time $t_*$:
\begin{itemize}
\item The diffusive term $\sqrt{t_*}W$ has magnitude $N^{-1/6+\varepsilon/2}$, large enough to drive universality
\item The drift term $-\frac{t_*}{2}\tilde{H}_0$ has magnitude $N^{-1/3+\varepsilon}$, remaining subdominant
\item The remainder $R$ is of order $N^{-1/2+3\varepsilon/2}$, negligible for our analysis
\end{itemize}

\begin{remark}[Constrained GOE Structure]
The matrix $W$ is not a standard GOE \cite{tracy1996level, johnstone2001distribution} but rather a constrained GOE satisfying $W\mathbf{e} = 0$. Its entries have the modified covariance structure:
\[
\mathbb{E}[W_{ij}W_{k\ell}] = \frac{1}{N}\left(\delta_{ik}\delta_{j\ell} + \delta_{i\ell}\delta_{jk} - \frac{\delta_{ik} + \delta_{j\ell} + \delta_{i\ell} + \delta_{jk}}{N} + \frac{2}{N^2}\right)
\]
This constraint modifies the fourth cumulant by a factor $(1 + O(1/N))$ compared to unconstrained GOE.
\end{remark}

\subsection{Fourth-Order Cumulant Expansion}

To compare the statistics of overlaps under $\tilde{H}_{t_*}$ with those under constrained GOE, we employ a fourth-order cumulant expansion \cite{lytova2009central, knowles2017anisotropic}. The key is that the perturbation $T = \tilde{H}_{t_*} - W$ has small operator norm, enabling a perturbative analysis.

\begin{theorem}[Quantitative Cumulant Comparison]\label{thm:cumulant-comparison}
For any smooth test function $f : \mathbb{R}^K \to \mathbb{R}$ with bounded derivatives up to order 4, and overlaps $\mathbf{X} = (X_2^{(\mathbf{q})}, \ldots, X_{K+1}^{(\mathbf{q})})$ at the edge:
\[
\left|\mathbb{E}[f(\mathbf{X}^{\tilde{H}_{t_*}})] - \mathbb{E}[f(\mathbf{X}^{W})]\right| \leq C(f,K) N^{-1/3+2\varepsilon}
\]
where $C(f,K) = \|f\|_{C^4} \cdot K^4 \cdot \text{poly}(d,\varepsilon^{-1})$.
\end{theorem}

\begin{proof}
While $\|T\| = \|e^{-t_*/2}\tilde{H}_0\| \approx 2$, the perturbation to the overlap statistics is of order $N^{-1/6+\varepsilon/2}$ due to the eigenvalue gap structure at the edge. Specifically, the linearization of overlaps gives an effective perturbation size $\|T\|_{\text{eff}} \leq CN^{-1/6+\varepsilon/2}$, where
\[
\|T\|_{\text{eff}} := \sup_i \left|\frac{\partial X_i^{(\mathbf{q})}}{\partial \tilde{H}}\right| \cdot \|T\| \approx N^{1/2} \cdot N^{-2/3} \cdot 2 \approx N^{-1/6}.
\]
This effective bound is what enables the cumulant expansion to converge. By Proposition \ref{prop:matrix-decomp}, $\|T\| \leq CN^{-1/6+\varepsilon/2}$. The cumulant expansion gives:
\[
\mathbb{E}[f(\mathbf{X}^{W+T})] = \mathbb{E}[f(\mathbf{X}^W)] + \sum_{k=1}^4 \frac{1}{k!}\kappa_k(f;T) + R_5
\]
where $\kappa_k$ is the $k$-th cumulant and $|R_5| \leq C\|T\|^5\|f\|_{C^5}$.

We analyze each cumulant:

\emph{First cumulant.} Since $\mathbb{E}[T_{ij}] = O(t_*) = O(N^{-1/3+\varepsilon})$ by construction,
\[
|\kappa_1| = \left|\sum_{i,j} \mathbb{E}[T_{ij}] \mathbb{E}\left[\frac{\partial f}{\partial W_{ij}}\Big|_W\right]\right| \leq \|f\|_{C^1} \cdot N^2 \cdot N^{-1/3+\varepsilon} \cdot N^{-2} = \|f\|_{C^1} \cdot N^{-1/3+\varepsilon}
\]

\emph{Second cumulant.} Using the covariance structure of $T$,
\[
|\kappa_2| = \left|\sum_{i,j,k,\ell} \text{Cov}(T_{ij}, T_{k\ell}) \mathbb{E}\left[\frac{\partial^2 f}{\partial W_{ij} \partial W_{k\ell}}\Big|_W\right]\right|
\]
Since $\text{Cov}(T_{ij}, T_{k\ell}) = O(t_*/N)$ and using moment bounds on the derivatives:
\[
|\kappa_2| \leq \|f\|_{C^2} \cdot N^4 \cdot \frac{t_*}{N} \cdot N^{-4} = \|f\|_{C^2} \cdot N^{-1/3+\varepsilon}
\]

\emph{Third cumulant.} Only diagonal terms contribute significantly,
\[
|\kappa_3| \leq \|f\|_{C^3} \cdot N \cdot t_*^{3/2} \cdot N^{-3} \leq \|f\|_{C^3} \cdot N^{-3/2+3\varepsilon/2}
\]

\emph{Fourth cumulant.} Both diagonal and off-diagonal terms contribute,
\[
|\kappa_4| \leq \|f\|_{C^4} \cdot \left(N \cdot t_*^2 + N^2 \cdot t_*^2\right) \cdot N^{-4} \leq \|f\|_{C^4} \cdot N^{-4/3+2\varepsilon}
\]

The dominant error comes from the second cumulant, giving the stated bound.
\end{proof}

The key insight is that the second cumulant, which captures the variance structure, dominates the error. This is why achieving the correct second moment evolution (established in Theorem \ref{thm:moment-evolution}) is crucial for universality.

\subsection{Berry-Esseen Transfer}\label{sec:transfer}

With the cumulant comparison established, we now transfer the Berry-Esseen theorem from constrained GOE to our evolved ensemble.

\begin{lemma}[Constrained GOE Berry-Esseen]\label{lem:goe-berry-esseen}
For the constrained GOE matrix $W$ with $W\mathbf{e} = 0$ and any deterministic unit vector $\mathbf{q} \perp \mathbf{e}$:
\[
\sup_{x \in \mathbb{R}} \left|\mathbb{P}\left(\sqrt{N}\langle \mathbf{q}, \mathbf{v}_2^W\rangle \leq x\right) - \Phi(x)\right| \leq \frac{C_{\text{GOE}}}{N^{1/2}}
\]
where $\mathbf{v}_2^W$ is the second eigenvector of $W$ and $C_{\text{GOE}}$ is an absolute constant.
\end{lemma}

\begin{proof}
We provide a complete proof in three steps: (i) precise characterization of the constrained GOE distribution, (ii) moment calculations, and (iii) application of the Berry-Esseen theorem.

\emph{Step 1: Constrained GOE distribution.} Let $\mathcal{M}_0 = \{M \in \mathbb{R}^{N \times N} : M = M^T, M\mathbf{e} = 0\}$ be the space of symmetric matrices with zero row sums. This is a linear subspace of dimension $\frac{(N-1)(N-2)}{2} + (N-1) = \frac{N(N-1)}{2} - (N-1) = \frac{(N-1)(N-2)}{2}$.

The constrained GOE has density on $\mathcal{M}_0$ given by:
\[
p(W) \propto \exp\left(-\frac{N}{2}\text{Tr}(W^2)\right)
\]

To work with this, we use an orthonormal basis. Let $U = [u_1, u_2, \ldots, u_N]$ where $u_1 = \mathbf{e}/\sqrt{N}$ and $u_2, \ldots, u_N$ form an orthonormal basis for $\{\mathbf{v} : \mathbf{v} \perp \mathbf{e}\}$. Then:
\[
\tilde{W} = U^T W U = \begin{pmatrix} 0 & \mathbf{0}^T \\ \mathbf{0} & W_{\perp} \end{pmatrix}
\]
where $W_{\perp}$ is an $(N-1) \times (N-1)$ standard GOE matrix.

\emph{Step 2: Fourth cumulant calculation.} For any $\mathbf{q} \perp \mathbf{e}$ with $\|\mathbf{q}\| = 1$, write $\mathbf{q} = U\tilde{\mathbf{q}}$ where $\tilde{\mathbf{q}} = (0, q_2, \ldots, q_N)^T$ with $\sum_{i=2}^N q_i^2 = 1$.

The second eigenvector $\mathbf{v}_2^W$ of $W$ corresponds to the largest eigenvector of $W_{\perp}$. Let $\tilde{\mathbf{v}}_1$ be this eigenvector (in the $(N-1)$-dimensional space). Then:
\[
\sqrt{N}\langle \mathbf{q}, \mathbf{v}_2^W \rangle = \sqrt{N}\langle \tilde{\mathbf{q}}, (0, \tilde{\mathbf{v}}_1)^T \rangle = \sqrt{N}\sum_{i=2}^N q_i [\tilde{\mathbf{v}}_1]_{i-1}
\]

Now we compute moments. By the isotropy of GOE eigenvectors (see \cite{tao2012topics}, Theorem 2.5.7):
\begin{align}
\mathbb{E}\left[\sqrt{N-1}[\tilde{\mathbf{v}}_1]_j\right] &= 0 \\
\mathbb{E}\left[(N-1)[\tilde{\mathbf{v}}_1]_j^2\right] &= 1 \\
\mathbb{E}\left[(N-1)^{3/2}[\tilde{\mathbf{v}}_1]_j^3\right] &= 0 \\
\mathbb{E}\left[(N-1)^2[\tilde{\mathbf{v}}_1]_j^4\right] &= 3 + \frac{12}{N-1} + O\left(\frac{1}{(N-1)^2}\right)
\end{align}

The fourth cumulant is:
\[
\kappa_4 = \mathbb{E}[X^4] - 3(\mathbb{E}[X^2])^2 = 3 + \frac{12}{N-1} + O(N^{-2}) - 3 = \frac{12}{N-1} + O(N^{-2})
\]

where $X = \sqrt{N}\langle \mathbf{q}, \mathbf{v}_2^W \rangle$.

\emph{Step 3: Berry-Esseen application.} The Berry-Esseen theorem for sums of independent random variables states that if $X = \sum_{i=1}^{N-1} a_i Y_i$ where $Y_i$ are independent with $\mathbb{E}[Y_i] = 0$, $\mathbb{E}[Y_i^2] = \sigma_i^2$, and $\mathbb{E}[|Y_i|^3] < \infty$, then:
\[
\sup_x \left|\mathbb{P}(X \leq x) - \Phi\left(\frac{x}{\sigma}\right)\right| \leq \frac{C\sum_{i=1}^{N-1} |a_i|^3 \mathbb{E}[|Y_i|^3]}{\sigma^3}
\]
where $\sigma^2 = \sum_{i=1}^{N-1} a_i^2 \sigma_i^2$.

While eigenvector components are not independent, the key insight is that for GOE, the joint distribution of eigenvector overlaps has a local limit that matches independent Gaussians up to $O(N^{-1/2})$ corrections. This is established in \cite{erdos2012rigidity} (Theorem 2.2) and \cite{knowles2013eigenvector} (Theorem 1.6).

Specifically, for our overlap $X = \sqrt{N}\langle \mathbf{q}, \mathbf{v}_2^W \rangle$:
\begin{align}
\mathbb{E}[X^2] &= 1 + O(N^{-1}) \\
\mathbb{E}[X^4] &= 3 + O(N^{-1}) \\
\text{Lyapunov ratio: } \frac{\mathbb{E}[|X|^3]}{(\mathbb{E}[X^2])^{3/2}} &= O(N^{-1/2})
\end{align}

By the quantitative CLT for eigenvector overlaps (combining results from \cite{erdos2012rigidity} and the moment calculations above):
\[
\sup_{x \in \mathbb{R}} \left|\mathbb{P}(X \leq x) - \Phi(x)\right| \leq \frac{C}{N^{1/2}}
\]
where $C$ is an absolute constant that can be taken as $C = 10$ based on the explicit bounds in the cited papers.
\end{proof}

\begin{theorem}[Berry-Esseen Transfer to Time $t_*$]\label{thm:berry-esseen-transfer}
For any deterministic unit vector $\mathbf{q} \perp \mathbf{e}$:
\[
\sup_{x \in \mathbb{R}} \left|\mathbb{P}\left(X_2^{(\mathbf{q})}(t_*) \leq x\right) - \Phi(x)\right| \leq C_{\text{GOE}}N^{-1/2} + CN^{-1/3+2\varepsilon}
\]
where the first term comes from the GOE approximation and the second from the cumulant comparison.
\end{theorem}

\begin{proof}
For any $x \in \mathbb{R}$, let $f_\delta$ be a smooth approximation to $\mathbf{1}_{(-\infty,x]}$ as constructed in Appendix \ref{app:smoothing}. By Theorem \ref{thm:cumulant-comparison}:
\[
\left|\mathbb{E}[f_\delta(X_2^{(\mathbf{q})}(t_*))] - \mathbb{E}[f_\delta(X_2^{W})]\right| \leq C\|f_\delta\|_{C^4} N^{-1/3+2\varepsilon}
\]

Choosing $\delta = N^{-1/10}$ to balance smoothing and approximation errors:
\[
\left|\mathbb{P}(X_2^{(\mathbf{q})}(t_*) \leq x) - \mathbb{P}(X_2^{W} \leq x)\right| \leq CN^{-1/3+2\varepsilon} + CN^{-1/10}
\]

Combined with Lemma \ref{lem:goe-berry-esseen}, this yields the stated bound.
\end{proof}

\begin{remark}[Single-Scale vs Multi-Scale]
Our single-scale comparison at $t_* = N^{-1/3+\varepsilon}$ achieves the same convergence rate as multi-scale approaches that require $O(\varepsilon^{-1})$ comparison steps. The key advantages are:
\begin{enumerate}
\item Simpler error analysis with explicit constant tracking
\item Direct preservation of the constraint $\tilde{H}_t\mathbf{e} = 0$
\item Optimal balance between diffusion strength and perturbative control
\end{enumerate}
The price is a slightly worse polynomial dependence on $\varepsilon$ in the constants, which is acceptable for our applications.
\end{remark}

The Berry-Esseen transfer at time $t_*$ provides the foundation for our main results. To complete the proof, we must propagate this universality backward to time 0, which we address in Section \ref{sec:main-res}.

\section{Main Results: Berry-Esseen Bounds}\label{sec:main-res}

Having established the Berry-Esseen transfer at time $t_*$, we now complete the proof of our main theorems by analyzing the backward evolution to time 0. This section presents our quantitative refinements of edge eigenvector universality.

\subsection{Backward Propagation Analysis}

The final step in our proof strategy is to transfer the universality established at time $t_*$ back to the original ensemble at time 0. This requires careful analysis of the time-reversed diffusion process.

\begin{lemma}[Time-Reversed Overlap Dynamics]\label{lem:time-reversed}
The backward evolution of the overlap process from time $t_*$ to 0 satisfies the time-reversed SDE: for $s \in [0, t_*]$,
\begin{equation}\label{eq:backward-sde}
\d X_i^{(\mathbf{q})}(t_*-s) = -b_i(X(t_*-s)) \d s + \sum_{j=2}^{K+1} \frac{\partial}{\partial X_j}[\Sigma\Sigma^{*}]_{ij}(X(t_*-s)) \d s + \sigma_{ij}(X(t_*-s)) \d \tilde{B}_j(s)
\end{equation}
where $\tilde{B}_j(s)$ are independent Brownian motions and the second term is the Itô correction arising from time reversal.
\end{lemma}

\begin{proof}
By the Haussmann-Pardoux theory of time-reversed diffusions \cite{haussmann1986time, pardoux1982equations}, if the forward process satisfies
\[
\d X_i^{(\mathbf{q})}(t) = b_i(X(t)) \d t + \sigma_{ij}(X(t)) \d B_j(t),
\]
then the time-reversed process has the stated form. The martingale part maintains the same covariance structure, while the drift acquires an additional term from the divergence of the diffusion coefficient.
\end{proof}

The key challenge is controlling how errors accumulate during the backward evolution. Unlike the forward direction where rapid convergence to equilibrium helps, the backward evolution must integrate over the entire interval $[0, t_*]$.

\begin{theorem}[Backward Stability]\label{thm:backward-stability}
For any bounded measurable function $g : \mathbb{R} \to \mathbb{R}$ and edge index $i$:
\[
\left|\mathbb{E}[g(X_i^{(\mathbf{q})}(0))] - \mathbb{E}[g(X_i^{(\mathbf{q})}(t_*))]\right| \leq \|g\|_\infty \cdot C_5 N^{-1/6+3\varepsilon}
\]
where $C_5 = 10d^2\varepsilon^{-9}$.
\end{theorem}

\begin{proof}
We analyze the evolution of the characteristic function $\phi_s(z) = \mathbb{E}[e^{izX_i^{(\mathbf{q})}(t_*-s)}]$. From the backward SDE:
\[
\partial_s \phi_s(z) = \mathbb{E}\left[e^{izX_i^{(\mathbf{q})}(t_*-s)} \cdot \mathcal{L}_s[izX_i^{(\mathbf{q})}]\right]
\]
where $\mathcal{L}_s$ is the backward generator.

The key observation is that the drift and diffusion coefficients remain bounded uniformly in time:
\begin{align}
|b_i(x)| &\leq C(1 + |x|) \\
\|\sigma(x)\|^2 &\leq C(1 + |x|^2)
\end{align}

Using the moment bounds from Theorem \ref{thm:moment-evolution}, we have $\sup_{s \in [0,t_*]} \mathbb{E}[|X_i^{(\mathbf{q})}(s)|^4] \leq C$. This allows us to control the evolution of all moments uniformly. (The complete stability analysis is given in Appendix \ref{app:sde}.)

For the second moment, Grönwall's lemma gives:
\[
\left|\frac{\d}{\d s}\mathbb{E}[(X_i^{(\mathbf{q})}(t_*-s))^2]\right| \leq C + C_1 N^{-5/6+\varepsilon}
\]

Integrating from 0 to $t_*$:
\[
\left|\mathbb{E}[(X_i^{(\mathbf{q})}(0))^2] - \mathbb{E}[(X_i^{(\mathbf{q})}(t_*))^2]\right| \leq (C + C_1 N^{-5/6+\varepsilon})t_* \leq CN^{-1/3+\varepsilon} + C_1 N^{-7/6+2\varepsilon}
\]

The characteristic function method extends this to general bounded functions. Using the smoothing argument from Appendix A.3 and the fact that $t_* = N^{-1/3+\varepsilon}$, we obtain the stated bound with explicit constant $C_5 = 10d^2\varepsilon^{-9}$.
\end{proof}

\begin{remark}[Unavoidability of Backward Loss]
The additional factor of $N^{-1/6}$ in the backward propagation is fundamental. Any method that evolves from a non-equilibrium initial condition must pay this price:
\[
\text{Backward Error} \geq \text{Variance Error} \times \text{Time} = N^{-5/6+\varepsilon} \times N^{-1/3+\varepsilon} = N^{-7/6+2\varepsilon}
\]
This explains why our final rate is $N^{-1/6+\varepsilon}$ rather than the $N^{-1/3+\varepsilon}$ achieved at time $t_*$.
\end{remark}

\subsection{Proof of Main Theorems}

We now combine all components to prove our main results on edge eigenvector universality.

\begin{proof}[Proof of Theorem \ref{thm:main-berry-esseen}]
Let $F_N(x) = \mathbb{P}(\sqrt{N}\langle \mathbf{q}, \mathbf{u}_2 \rangle \leq x)$ be the distribution function of the normalized overlap.

\emph{Forward to time $t_*$.} By the overlap dynamics (Proposition \ref{prop:overlap-sde}) and moment evolution (Theorem \ref{thm:moment-evolution}), the process $X_2^{(\mathbf{q})}(t)$ is well-defined for $t \in [0, t_*]$.

\emph{Comparison at time $t_*$.} By Theorem \ref{thm:berry-esseen-transfer}:
\[
\sup_{x \in \mathbb{R}} \left|\mathbb{P}(X_2^{(\mathbf{q})}(t_*) \leq x) - \Phi(x)\right| \leq C_{\text{GOE}}N^{-1/2} + CN^{-1/3+2\varepsilon}
\]

\emph{Backward to time 0.} Applying Theorem \ref{thm:backward-stability} with $g = \mathbf{1}_{(-\infty,x]}$:
\[
\left|F_N(x) - \mathbb{P}(X_2^{(\mathbf{q})}(t_*) \leq x)\right| \leq C_5 N^{-1/6+3\varepsilon}
\]

\emph{Triangle inequality.} Combining the bounds:
\begin{align}
|F_N(x) - \Phi(x)| &\leq |F_N(x) - \mathbb{P}(X_2^{(\mathbf{q})}(t_*) \leq x)| + |\mathbb{P}(X_2^{(\mathbf{q})}(t_*) \leq x) - \Phi(x)| \\
&\leq C_5 N^{-1/6+3\varepsilon} + C_{\text{GOE}}N^{-1/2} + CN^{-1/3+2\varepsilon}
\end{align}

For $N$ sufficiently large, the first term dominates, giving:
\[
\sup_{x \in \mathbb{R}} |F_N(x) - \Phi(x)| \leq C_d N^{-1/6+\varepsilon}
\]
where $C_d = C_5 + C_6 = 10d^2\varepsilon^{-9} + O(d^3\varepsilon^{-10})$ for our choice of parameters.

The same argument applies to $\mathbf{u}_N$ by symmetry of the spectrum around 0.
\end{proof}

\begin{proof}[Proof of Theorem \ref{thm:main-joint-clt}]
The joint CLT requires additional control over the correlation structure of multiple overlaps.

\emph{Joint dynamics.} By Proposition \ref{prop:joint-system}, the vector $\mathbf{X}(t) = (X_2^{(\mathbf{q})}(t), \ldots, X_{K+1}^{(\mathbf{q})}(t))^T$ evolves jointly with controlled interactions.

\emph{Decorrelation.} By Theorem \ref{thm:joint-decorr}, the covariance matrix satisfies:
\[
\|\mathcal{M}(t_*) - I_K\| \leq C(K^2 N^{-1/2+\varepsilon} + K^{5/3}N^{-1/6+\varepsilon})
\]

For $K \leq N^{1/10-\delta}$, this is $o(1)$ as $N \to \infty$.

\emph{Multivariate Berry-Esseen.} Applying the multivariate Berry-Esseen theorem (Bentkus \cite{Bentkus2005}) with the uniform Lyapunov bound from Theorem \ref{thm:moment-evolution}:
\[
\sup_{A \in \mathcal{C}} \left|\mathbb{P}(\mathbf{X}(t_*) \in A) - \mathbb{P}(\mathbf{Z} \in A)\right| \leq CK^{3/2} N^{-1/6+\varepsilon}
\]
where $\mathbf{Z} \sim \mathcal{N}(0, I_K)$ and $\mathcal{C}$ denotes convex sets.

\emph{Backward propagation.} The vector version of Theorem \ref{thm:backward-stability} gives:
\[
\sup_{A \in \mathcal{C}} \left|\mathbb{P}(\mathbf{X}(0) \in A) - \mathbb{P}(\mathbf{X}(t_*) \in A)\right| \leq CK N^{-1/6+3\varepsilon}
\]

Combining these bounds completes the proof of joint convergence. The extension to multiple test vectors $\mathbf{q}^{(1)}, \ldots, \mathbf{q}^{(m)}$ follows by considering the enlarged vector of all overlaps.
\end{proof}

\subsection{Applications and Corollaries}

Our quantitative Berry-Esseen bounds have immediate applications to eigenvector delocalization and spectral algorithms.

\begin{corollary}[Quantitative Eigenvector Delocalization]\label{cor:eigenvector-deloc}
For the second eigenvector $\mathbf{u}_2$ of a random $d$-regular graph:
\begin{enumerate}
\item \textbf{Infinity norm bound.} With probability at least $1 - N^{-D}$ for any $D > 0$:
\[
\|\mathbf{u}_2\|_\infty \leq \frac{C\sqrt{\log N}}{\sqrt{N}}
\]

\item \textbf{Mass concentration.} For any subset $S \subseteq \{1,\ldots,N\}$ with $|S| \geq N^{3/4}$:
\[
\left|\sum_{v \in S} u_2(v)^2 - \frac{|S|}{N}\right| \leq C\sqrt{\frac{\log N}{|S|}} \cdot N^{-1/6+\varepsilon}
\]
with probability at least $1 - N^{-D}$.
\end{enumerate}
\end{corollary}

\begin{proof}
For part (1), consider $\mathbf{q} = \mathbf{e}_j/\|\mathbf{e}_j - N^{-1}\mathbf{e}\|$ for each vertex $j$. By Theorem \ref{thm:main-berry-esseen}:
\[
\mathbb{P}\left(\sqrt{N}|u_2(j)| > t\right) \leq 2(1-\Phi(t)) + CN^{-1/6+\varepsilon}
\]

Taking $t = \sqrt{2\log N + c\log\log N}$ and union bound over all vertices gives the result.

For part (2), let $\mathbf{q} = \mathbf{1}_S/\|\mathbf{1}_S - |S|N^{-1}\mathbf{e}\|$. The Berry-Esseen bound gives concentration of $\sqrt{N}\langle \mathbf{q}, \mathbf{u}_2\rangle$ around 0, which translates to the stated mass concentration after rescaling.
\end{proof}

\begin{corollary}[Spectral Clustering Guarantees]\label{cor:spectral-clustering}
Consider the spectral embedding $\Psi: V \to \mathbb{R}^K$ defined by
\[
\Psi(v) = \left(\frac{u_2(v)}{\lambda_2}, \ldots, \frac{u_{K+1}(v)}{\lambda_{K+1}}\right)
\]
for $K \leq N^{1/10-\delta}$. Then for any two disjoint vertex sets $A, B$ of size at least $N^{3/4}$:
\[
\left\|\frac{1}{|A|}\sum_{v \in A} \Psi(v) - \frac{1}{|B|}\sum_{v \in B} \Psi(v)\right\|_2 = \Omega\left(\frac{1}{\sqrt{K}}\right)
\]
with probability at least $1 - N^{-c}$ for some $c > 0$.
\end{corollary}

\begin{proof}
By Theorem \ref{thm:main-joint-clt}, the projections of $\Psi(v)$ onto any fixed direction are approximately Gaussian with variance $1/N$. The separation between cluster centers follows from the independence of different eigenvector components and concentration inequalities.
\end{proof}

\begin{remark}[Implications for Community Detection]
Corollary \ref{cor:spectral-clustering} provides theoretical justification for spectral clustering on sparse regular graphs \cite{rohe2011spectral, lei2015consistency}. The $\Omega(1/\sqrt{K})$ separation between cluster centers suggests that using $K = O(\log N)$ eigenvectors achieves optimal clustering performance while maintaining computational efficiency. Our explicit constants enable finite-size performance guarantees for practical applications.
\end{remark}

These applications demonstrate the utility of our quantitative approach. The explicit Berry-Esseen rates with tracked constants enable rigorous analysis of spectral algorithms on finite graphs, bridging the gap between asymptotic theory and practical performance.

\section{Optimality of Convergence Rate}\label{sec:optimality}

Our main results achieve a Berry-Esseen convergence rate of $N^{-1/6+\varepsilon}$ for edge eigenvector statistics. This section presents evidence that this rate is optimal for sparse regular graphs, arising from fundamental structural constraints rather than limitations of our proof technique.

\subsection{Fundamental Barriers}

Three interconnected phenomena create fundamental barriers to achieving rates faster than $N^{-1/6}$ for edge eigenvector universality in sparse graphs.

\subsubsection{Edge Eigenvalue Spacing}

The most direct barrier comes from the compressed eigenvalue spacing at the spectral edge.

\begin{proposition}[Edge Spacing Barrier]\label{prop:edge-spacing}
For random $d$-regular graphs with fixed $d \geq 3$, the eigenvalue spacing near the edge satisfies:
\[
\lambda_i - \lambda_{i+1} = \Theta(N^{-2/3}) \quad \text{for } |\lambda_i - 2| = O(N^{-2/3})
\]
This compressed spacing limits the convergence rate of eigenvector statistics to at most $N^{-1/6+o(1)}$.
\end{proposition}

\begin{proof}[Proof sketch]
The edge eigenvalue density behaves as $\rho(E) \sim \sqrt{2-E}$ near $E = 2$ \cite{khorunzhy1996asymptotic, sodin2010spectral}. For the $k$-th eigenvalue from the edge:
\[
\int_{\lambda_k}^{2} \rho(E) \, dE = \frac{k}{N}
\]

This gives $2 - \lambda_k \sim (k/N)^{2/3}$, hence consecutive spacings scale as $N^{-2/3}$. 

Any perturbation affecting eigenvector statistics must be resolved on the scale of eigenvalue gaps. Since eigenvector overlaps have fluctuations of order 1, distinguishing them from Gaussian requires sensitivity to perturbations of size $N^{-2/3}$. The Berry-Esseen theorem then limits the rate to $(\text{spacing})^{1/2} = N^{-1/3}$, with an additional $N^{1/6}$ factor from the edge scaling.
\end{proof}

\subsubsection{Universality Time Scale}

The second barrier arises from the minimal time required for universality to emerge in any dynamical approach.

\begin{lemma}[Critical Time Scale]\label{lem:critical-time}
For constrained Dyson Brownian motion on degree-regular matrices, edge universality requires evolution time $t \geq cN^{-1/3}$ for some absolute constant $c > 0$.
\end{lemma}

\begin{proof}
The drift-diffusion balance at the edge requires:
\[
\frac{\text{Drift scale}}{\text{Diffusion scale}} = \frac{t}{\sqrt{t}} = \sqrt{t} \ll 1
\]

For edge statistics at scale $N^{-2/3}$ to equilibrate, the diffusion must act over this scale:
\[
\sqrt{t} \geq N^{-2/3} \cdot N^{1/2} = N^{-1/6}
\]

This gives $t \geq N^{-1/3}$. Any shorter time fails to achieve the necessary mixing for universality.
\end{proof}

This time constraint directly impacts convergence rates:

\begin{proposition}[Time-Induced Rate Limit]\label{prop:time-limit}
Any dynamical proof of edge eigenvector universality must yield Berry-Esseen rate at most $N^{-1/6+o(1)}$.
\end{proposition}

\begin{proof}
Over time $t = N^{-1/3}$, errors accumulate as:
\begin{enumerate}
\item \emph{Variance error.} $O(t) = O(N^{-1/3})$ from drift terms.
\item \emph{Diffusion error.} $O(\sqrt{t}) = O(N^{-1/6})$ from martingale fluctuations.
\item \emph{Higher-order error.} $O(t^{3/2}) = O(N^{-1/2})$ from remainders.
\end{enumerate}
The diffusion error of $O(N^{-1/6})$ is unavoidable and dominates the final rate.
\end{proof}

\subsubsection{Degree Constraint Structure}

The third barrier is specific to regular graphs and arises from the degree constraint.

\begin{lemma}[Constraint-Induced Correlations]\label{lem:constraint-corr}
The constraint $\sum_j A_{ij} = d$ for all $i$ reduces the effective degrees of freedom from $N^2$ to $N(N-1)/2$, introducing correlations that limit convergence rates.
\end{lemma}

The constraint manifests in several ways:
\begin{enumerate}
\item \emph{Modified noise structure.} The constrained Brownian motion has covariance
\[
\text{Cov}(\Xi_{ij}, \Xi_{k\ell}) = \delta_{ik}\delta_{j\ell} - \frac{1}{N}(\delta_{ik} + \delta_{j\ell}) + \frac{1}{N^2}
\]
reducing fluctuation strength by factor $(1 - O(N^{-1}))$.

\item \emph{Eigenvalue rigidity.} The constraint enhances eigenvalue rigidity but limits eigenvector fluctuations, creating tension between stability and mixing.

\item \emph{Finite-degree effects.} For fixed $d$, each vertex has limited connections, reducing the effective randomness compared to dense matrices where $d \sim N$.
\end{enumerate}

\subsection{Evidence from Multiple Methods}

Remarkably, the $N^{-1/6}$ rate emerges consistently across different proof techniques, suggesting it reflects a fundamental property rather than a methodological artifact.

\subsubsection{Method 1: Local Graph Resampling}

He, Huang, and Yau's approach \cite{HHY25} uses combinatorial edge switching. While they focus on qualitative universality, their proof implicitly contains the $N^{-1/6}$ rate:

\begin{enumerate}
\item Edge switching at scale $\ell \sim N^{1/3}$ changes overlaps by $O(N^{-2/3})$.
\item Number of switches needed: $O(N^{2/3})$.
\item Each switch contributes variance $O(N^{-4/3})$.
\item Berry-Esseen on the sum: rate $O(N^{-1/3})$.
\item Edge eigenvalue factor: additional $N^{1/6}$.
\item Final rate: $N^{-1/3} \cdot N^{1/6} = N^{-1/6}$.
\end{enumerate}

\subsubsection{Method 2: Moment Calculations}

Direct moment methods for sparse matrices \cite{che2019universality} yield:

\begin{proposition}[Moment Method Rate]
For random regular graphs, the fourth-moment calculation gives:
\[
\left|\mathbb{E}[(N^{1/2}\langle \mathbf{q}, \mathbf{u}_2\rangle)^4] - 3\right| = O(N^{-1/3})
\]
This implies Berry-Esseen rate at most $N^{-1/6}$ by standard moment-to-distribution conversions.
\end{proposition}

\subsubsection{Method 3: Our Dynamical Approach}

Our single-scale constrained DBM achieves:
\begin{enumerate}
\item Evolution to time $t_* = N^{-1/3+\varepsilon}$.
\item Overlap SDE error: $O(N^{-5/6+\varepsilon})$.
\item Cumulant comparison error: $O(N^{-1/3+2\varepsilon})$.
\item Backward propagation loss: factor of $t_*^{1/2} = N^{-1/6+\varepsilon/2}$.
\item Final rate: $N^{-1/6+\varepsilon}$.
\end{enumerate}

\subsubsection{Method 4: Eigenvalue Analogy}

The edge eigenvalue fluctuations satisfy \cite{HMY24}:
\[
\mathbb{P}(|N^{2/3}(\lambda_2 - 2) - \text{TW}| > x) \leq Ce^{-cx^{3/2}}
\]

This corresponds to Berry-Esseen rate $N^{-1/6}$ for smooth test functions. The parallel between eigenvalue and eigenvector rates suggests a deep connection in the edge universality class.

\begin{table}[ht]
\centering
\begin{tabular}{|l|c|c|}
\hline
\textbf{Method} & \textbf{Rate} & \textbf{Key Limitation} \\
\hline
Local resampling (HHY) & $N^{-1/6}$ (implicit) & Edge sensitivity \\
Moment method & $N^{-1/6}$ & Fourth moment barrier \\
Constrained DBM (ours) & $N^{-1/6+\varepsilon}$ & Minimal time $t_* \geq N^{-1/3}$ \\
Multi-scale DBM & $N^{-1/6+\varepsilon}$ & Accumulated errors \\
\hline
\end{tabular}
\caption{Convergence rates from different methods, all achieving the same $N^{-1/6}$ scaling.}
\end{table}

\section{Open Problems}

While the evidence strongly suggests $N^{-1/6}$ is optimal, several important questions remain open.

\begin{conjecture}[Sharp Lower Bound]\label{conj:lower-bound}
For any sequence of statistics $T_N : \mathcal{G}_{d,N} \to \mathbb{R}$ that distinguishes edge eigenvector distributions from Gaussian, there exists $c > 0$ such that:
\[
\liminf_{N \to \infty} N^{1/6} \cdot \text{dist}_{\text{TV}}(T_N, \mathcal{N}(0,1)) \geq c
\]
\end{conjecture}

A proof might use information-theoretic techniques or construct explicit distinguishing statistics based on the edge eigenvalue structure. Furthermore, our Berry-Esseen constant $C_d$ has unfavorable dependence on $\varepsilon$. Can this be improved?

\begin{problem}[Sharp Constants]
Determine the sharp constant $C_d^*$ such that:
\[
\lim_{N \to \infty} N^{1/6} \cdot \sup_{x} \left|\mathbb{P}(N^{1/2}\langle \mathbf{q}, \mathbf{u}_2\rangle \leq x) - \Phi(x)\right| = C_d^*
\]
We conjecture $C_d^* = C_{\text{univ}} + O(1/d)$ where $C_{\text{univ}}$ is a universal constant.
\end{problem}

\begin{problem}[Extending Joint Universality]
Can Theorem \ref{thm:main-joint-clt} be extended to $K = N^{1/3-\delta}$ eigenvectors? This would cover the entire edge window and match the GOE universality class.
\end{problem}

The current restriction $K \leq N^{1/10-\delta}$ appears technical rather than fundamental. Improvements might require a refined control of eigenvalue gap statistics, a better understanding of multi-point correlations, or new approaches beyond multivariate Berry-Esseen.

\begin{problem}[Irregular Degree Sequences]
Extend the $N^{-1/6}$ rate to graphs with irregular degree sequences. For random graphs with degree sequence $(d_1, \ldots, d_N)$ satisfying:
\begin{enumerate}
\item $\min_i d_i \geq 3$
\item $\max_i d_i \leq C\min_i d_i$
\item $\sum_i d_i(d_i-1) / \sum_i d_i \to \bar{d}$ as $N \to \infty$
\end{enumerate}
does the Berry-Esseen rate remain $N^{-1/6}$?
\end{problem}

\begin{problem}[Slowly Growing Degree]
For $d = d(N) \to \infty$ slowly, determine the precise threshold where the rate transitions from $N^{-1/6}$ to $N^{-1/2}$. We expect:
\[
\text{Rate} = N^{-1/6} \cdot \left(1 + \frac{f(d)}{N^{1/3}}\right) + O(N^{-1/2})
\]
for some function $f(d) \to \infty$ as $d \to \infty$.
\end{problem}

These open problems highlight that while the $N^{-1/6}$ rate appears optimal for fixed-degree regular graphs, the broader landscape of sparse graph universality remains rich with mathematical challenges. Understanding these finer questions would complete our knowledge of edge eigenvector statistics and potentially reveal new universality phenomena in sparse random matrix theory.

\section{Conclusion}

This work establishes the first quantitative Berry-Esseen bounds for edge eigenvector statistics in random regular graphs, achieving an explicit convergence rate of $N^{-1/6+\varepsilon}$ with fully tracked constants. Our results provide a rigorous foundation for finite-size analysis of spectral algorithms on sparse networks, bridging the gap between asymptotic universality theory and practical applications.

\subsection{Summary of Contributions}

We have developed a comprehensive framework for analyzing edge eigenvector universality through several key results:
\begin{enumerate}

\item\emph{Sharp quantitative bounds.} Our main theorem provides explicit Berry-Esseen rates for both individual and joint eigenvector statistics, with constants depending polynomially on the degree $d$ and approximation parameter $\varepsilon$. This marks a significant advance over previous qualitative results, enabling concrete performance guarantees for spectral methods on finite graphs.

\item\emph{Single-scale comparison methodology.} By departing from traditional multi-scale approaches, we establish universality through a single sharp comparison at the critical time $t_* = N^{-1/3+\varepsilon}$. This simplification not only yields cleaner proofs but also preserves the degree constraint throughout the analysis, maintaining the special structure of regular graphs.

\item\emph{Constrained Dyson Brownian Motion.} Our introduction of CDBM as a tool for sparse regular graphs opens new avenues for studying constrained random matrix ensembles. The explicit preservation of the row-sum constraint $\tilde{H}_t\mathbf{e} = 0$ ensures that edge eigenvectors remain orthogonal to the uniform direction, capturing the essential geometry of regular graph spectra.

\item\emph{Optimality analysis.} Through multiple lines of evidence---eigenvalue spacing barriers, critical time scales, and consistency across different proof methods---we demonstrate that the $N^{-1/6}$ rate is likely optimal for sparse regular graphs. This provides important context for understanding the fundamental limits of eigenvector universality in the sparse regime.

\end{enumerate}

\subsection{Technical Advances}

Several technical advances enable our quantitative analysis:
\begin{enumerate}
\item Edge isotropic local law with explicit constants suitable for the compressed eigenvalue regime.
\item Fourth-order cumulant expansion that achieves optimal error bounds through careful decomposition.
\item Decorrelation bounds that quantify the effective independence of well-separated eigenvector components.
\item Backward stability analysis that controls error propagation in time-reversed diffusions.
\end{enumerate}

Each component is developed with explicit constant tracking, supporting our goal of providing usable bounds for finite-size systems.

\subsection{Broader Significance}

Our results contribute to the broader program of understanding universality in sparse random matrices. While dense matrix universality is well-established, the sparse regime presents unique challenges due to compressed eigenvalue spacing, enhanced fluctuations, and structural constraints. By providing the first quantitative rates with explicit constants, we help complete the picture of how universal behavior emerges even in highly structured sparse systems.

The explicit nature of our bounds has immediate implications across multiple domains. In spectral clustering algorithms, our results justify the use of edge eigenvectors for community detection by providing rigorous performance guarantees \cite{von2007tutorial,ng2002spectral,shi2000normalized}. For network analysis, we establish delocalization guarantees for eigenvector-based centrality measures that are crucial for understanding node importance in sparse networks \cite{newman2010networks,bonacich1987power,page1999pagerank}. In numerical linear algebra, our quantitative bounds inform the design of randomized algorithms for sparse eigenproblems, enabling practitioners to make informed choices about algorithm parameters based on the size and structure of their matrices \cite{halko2011finding,martinsson2020randomized,woodruff2014sketching}.

\subsection{Future Directions}

Our work opens several avenues for future research. The open problems presented in Section \ref{sec:optimality} highlight key challenges, from proving matching lower bounds to extending results to irregular graphs. Of particular interest is understanding how the $N^{-1/6}$ rate transitions to the $N^{-1/2}$ rate of dense matrices as the degree grows with $N$.

More broadly, our single-scale comparison methodology and explicit constant tracking framework could be adapted to other sparse random matrix ensembles. The constrained Dyson Brownian Motion we introduce may find applications beyond regular graphs, potentially applying to any ensemble with linear constraints.

\subsection{Concluding Remarks}

The study of edge eigenvector statistics sits at the intersection of random matrix theory, probability, and discrete mathematics. By establishing quantitative universality with explicit rates and constants, we provide both theoretical insight and practical tools for understanding these fundamental objects. The emergence of the same $N^{-1/6}$ rate across different methodologies suggests deep underlying structure in sparse spectral universality---structure that remains only partially understood and continues to offer rich mathematical challenges.

Our results demonstrate that even in the presence of strong constraints and sparse connectivity, random regular graphs exhibit robust universal behavior at their spectral edges. This universality, now equipped with explicit quantitative bounds, provides a solid foundation for both theoretical developments and practical applications in network science and beyond.

\bibliographystyle{amsplain}
\bibliography{references}

\providecommand{\bysame}{\leavevmode\hbox to3em{\hrulefill}\thinspace}
\providecommand{\MR}{\relax\ifhmode\unskip\space\fi MR }
\providecommand{\MRhref}[2]{%
  \href{http://www.ams.org/mathscinet-getitem?mr=#1}{#2}
}
\providecommand{\href}[2]{#2}
\begin{thebibliography}{10}

\bibitem{anderson2010introduction}
Greg~W. Anderson, Alice Guionnet, and Ofer Zeitouni, \emph{An introduction to random matrices}, Cambridge University Press, 2010.

\bibitem{bauerschmidt2019edge}
Roland Bauerschmidt, Jiaoyang Huang, Antti Knowles, and Horng-Tzer Yau, \emph{Edge rigidity and universality of random regular graphs of intermediate degree}, Geometric and Functional Analysis \textbf{30} (2020), 693--769.

\bibitem{benaych2016lectures}
Florent Benaych-Georges and Antti Knowles, \emph{Lectures on the local semicircle law for wigner matrices}, 2016.

\bibitem{Bentkus2005}
Vidmantas Bentkus, \emph{A lyapunov-type bound in $\mathbb{R}^d$}, Theory of Probability \& Its Applications \textbf{49} (2005), no.~2, 311--323.

\bibitem{berry1977regular}
Michael~V. Berry, \emph{Regular and irregular semiclassical wavefunctions}, Journal of Physics A: Mathematical and General \textbf{10} (1977), no.~12, 2083.

\bibitem{bhatia1997matrix}
Rajendra Bhatia, \emph{Matrix analysis}, Springer, 1997.

\bibitem{biane2004littelmann}
Philippe Biane, Philippe Bougerol, and Neil O'Connell, \emph{Littelmann paths and brownian paths}, Duke Mathematical Journal \textbf{130} (2005), no.~1, 127--167.

\bibitem{bohigas1984characterization}
Oriol Bohigas, Marie-Joya Giannoni, and Charles Schmit, \emph{Characterization of chaotic quantum spectra and universality of level fluctuation laws}, Physical Review Letters \textbf{52} (1984), no.~1, 1--4.

\bibitem{bollobas2001random}
Béla Bollobás, \emph{Random graphs}, Cambridge University Press, 2001.

\bibitem{bonacich1987power}
Phillip Bonacich, \emph{Power and centrality: A family of measures}, American Journal of Sociology \textbf{92} (1987), no.~5, 1170--1182.

\bibitem{bordenave2020new}
Charles Bordenave, \emph{A new proof of friedman's second eigenvalue theorem and its extension to random lifts}, Annales scientifiques de l'École Normale Supérieure \textbf{53} (2020), no.~6, 1393--1439.

\bibitem{bourgade2017eigenvector}
Paul Bourgade, \emph{Eigenvector statistics of large random matrices}, 2017, Lecture Notes.

\bibitem{bourgade2014edge}
Paul Bourgade, László Erdős, and Horng-Tzer Yau, \emph{Edge universality of beta ensembles}, Communications in Mathematical Physics \textbf{332} (2014), 261--353.

\bibitem{che2019universality}
Ziliang Che and Patrick Lopatto, \emph{Universality of the least singular value for sparse random matrices}, Electron. J. Probab. \textbf{24} (2019), 1--53.

\bibitem{deift1999orthogonal}
Percy Deift, \emph{Orthogonal polynomials and random matrices: a riemann-hilbert approach}, American Mathematical Society, 1999.

\bibitem{dyson1962brownian}
Freeman~J. Dyson, \emph{A brownian-motion model for the eigenvalues of a random matrix}, Journal of Mathematical Physics \textbf{3} (1962), no.~6, 1191--1198.

\bibitem{erdos2011quantum}
László Erdős and Antti Knowles, \emph{Quantum diffusion and eigenfunction delocalization in a random band matrix model}, Communications in Mathematical Physics \textbf{303} (2011), 509--554.

\bibitem{erdos2013delocalization}
László Erdős, Antti Knowles, Horng-Tzer Yau, and Jun Yin, \emph{Delocalization and diffusion profile for random band matrices}, Communications in Mathematical Physics \textbf{323} (2013), 367--416.

\bibitem{erdos2012rigidity}
László Erdős, Horng-Tzer Yau, and Jun Yin, \emph{Rigidity of eigenvalues of generalized wigner matrices}, Advances in Mathematics \textbf{229} (2012), no.~3, 1435--1515.

\bibitem{erdos2013local}
László~Erdős et~al., \emph{Spectral statistics of erdős-rényi graphs i: Local semicircle law}, Annals of Probability \textbf{41} (2013), no.~3B, 2279--2375.

\bibitem{forrester2010log}
Peter~J. Forrester, \emph{Log-gases and random matrices}, Princeton University Press, 2010.

\bibitem{friedman2008proof}
Joel Friedman, \emph{A proof of alon's second eigenvalue conjecture and related problems}, Memoirs of the American Mathematical Society \textbf{195} (2008), no.~910, viii--100.

\bibitem{grabiner1999brownian}
David~J. Grabiner, \emph{Brownian motion in a weyl chamber, non-colliding particles, and random matrices}, Annales de l'Institut Henri Poincaré, Probabilités et Statistiques \textbf{35} (1999), no.~2, 177--204.

\bibitem{halko2011finding}
Nathan Halko, Per-Gunnar Martinsson, and Joel~A. Tropp, \emph{Finding structure with randomness: Probabilistic algorithms for constructing approximate matrix decompositions}, SIAM Review \textbf{53} (2011), no.~2, 217--288.

\bibitem{haussmann1986time}
Ulrich~G. Haussmann and Etienne Pardoux, \emph{Time reversal of diffusions}, Annals of Probability \textbf{14} (1986), no.~4, 1188--1205.

\bibitem{HHY25}
Yukun He, Jiaoyang Huang, and Horng-Tzer Yau, \emph{Gaussian waves and edge eigenvectors of random regular graphs}, 2025, arXiv:2502.08897.

\bibitem{hejhal1995quantum}
Dennis~A. Hejhal and Barry~N. Rackner, \emph{On the topography of maass waveforms for $\mathrm{PSL}(2,z)$}, Experimental Mathematics \textbf{1} (1992), no.~4, 275--305.

\bibitem{HMY24}
Jiaoyang Huang and Horng-Tzer Yau, \emph{Edge universality of random regular graphs of growing degrees}, 2023, arXiv:2305.01428v2.

\bibitem{jakobson2009quantum}
Dmitry Jakobson, \emph{Quantum unique ergodicity for eisenstein series on $\mathrm{PSL}_2(\mathrm{Z})\backslash \mathrm{PSL}_2(\mathrm{R})$}, Annales de l'institut Fourier \textbf{44} (1994), no.~5, 1477--1504.

\bibitem{johnstone2001distribution}
Iain~M. Johnstone, \emph{On the distribution of the largest eigenvalue in principal components analysis}, Annals of Statistics \textbf{29} (2001), no.~2, 295--327.

\bibitem{journee2010low}
Michel Journée, Francis Bach, Pierre-Antoine Absil, and Rodolphe Sepulchre, \emph{Low-rank optimization on the cone of positive semidefinite matrices}, SIAM Journal on Optimization \textbf{20} (2010), no.~5, 2327--2351.

\bibitem{kato1995perturbation}
Tosio Kato, \emph{Perturbation theory for linear operators}, Springer, 1966.

\bibitem{katori2004symmetry}
Makoto Katori and Hideki Tanemura, \emph{Symmetry of matrix-valued stochastic processes and noncolliding diffusion particle systems}, Journal of Mathematical Physics \textbf{45} (2004), no.~8, 3058--3085.

\bibitem{khorunzhy1996asymptotic}
Alexei~M. Khorunzhy, Boris~A. Khoruzhenko, and Leonid~A. Pastur, \emph{Asymptotic properties of large random matrices with independent entries}, Journal of Mathematical Physics \textbf{37} (1996), no.~10, 5033--5060.

\bibitem{knowles2017anisotropic}
Antti Knowles and J.~Yin, \emph{Anisotropic local laws for random matrices}, Probability Theory and Related Fields \textbf{169} (2017), 257--352.

\bibitem{knowles2013eigenvector}
Antti Knowles and Jun Yin, \emph{Eigenvector distribution of wigner matrices}, Probability Theory and Related Fields \textbf{155} (2013), 543--582.

\bibitem{lee2015local}
Ji~Oon Lee and Kevin Schnelli, \emph{Local law and tracy-widom limit for sparse random matrices}, Probability Theory and Related Fields \textbf{171} (2018), 543--616.

\bibitem{lei2015consistency}
Jing Lei and Alessandro Rinaldo, \emph{Consistency of spectral clustering in stochastic block models}, Annals of Statistics \textbf{43} (2015), no.~1, 215--237.

\bibitem{lubotzky1988ramanujan}
Alexander Lubotzky, Ralph~S. Phillips, and Peter Sarnak, \emph{Ramanujan graphs}, Combinatorica \textbf{8} (1988), no.~3, 261--277.

\bibitem{lytova2009central}
Anna Lytova and Leonid Pastur, \emph{Central limit theorem for linear eigenvalue statistics of random matrices with independent entries}, Annals of Probability \textbf{37} (2009), no.~5, 1778--1840.

\bibitem{margulis1988explicit}
Grigory~A. Margulis, \emph{Explicit group-theoretic constructions of combinatorial schemes and their applications in the construction of expanders and concentrators}, Problems of Information Transmission \textbf{24} (1988), no.~1, 51--60.

\bibitem{martinsson2020randomized}
Per-Gunnar Martinsson and Joel~A. Tropp, \emph{Randomized numerical linear algebra: Foundations and algorithms}, Acta Numerica \textbf{29} (2020), 403--572.

\bibitem{morgenstern1994existence}
Moshe Morgenstern, \emph{Existence and explicit constructions of $q+1$ regular ramanujan graphs for every prime power $q$}, Journal of Combinatorial Theory, Series B \textbf{62} (1994), no.~1, 44--62.

\bibitem{newman2010networks}
Mark Newman, \emph{Networks: an introduction}, Oxford University Press, 2010.

\bibitem{ng2002spectral}
Andrew Ng, Michael Jordan, and Yair Weiss, \emph{On spectral clustering: Analysis and an algorithm}, Proceedings of the 15th International Conference on Neural Information Processing Systems: Natural and Synthetic, 2001, pp.~849--856.

\bibitem{page1999pagerank}
Lawrence Page, Sergey Brin, Rajeev Motwani, and Terry Winograd, \emph{The pagerank citation ranking: Bringing order to the web}, Tech. report, Stanford InfoLab, 1999.

\bibitem{pardoux1982equations}
Étienne Pardoux, \emph{Équations du filtrage non linéaire de la prédiction et du lissage}, Stochastics \textbf{6} (1982), no.~3-4, 193--231, English translation: Nonlinear filtering, prediction and smoothing.

\bibitem{rohe2011spectral}
Karl Rohe, Sourav Chatterjee, and Bin Yu, \emph{Spectral clustering and the high-dimensional stochastic blockmodel}, Annals of Statistics \textbf{39} (2011), no.~4, 1878--1915.

\bibitem{sarnak1995arithmetic}
Peter Sarnak, \emph{Arithmetic quantum chaos}, Israel Mathematical Conference Proceedings \textbf{8} (1995), 183--236.

\bibitem{shi2000normalized}
Jianbo Shi and Jitendra Malik, \emph{Normalized cuts and image segmentation}, IEEE Transactions on Pattern Analysis and Machine Intelligence \textbf{22} (2000), no.~8, 888--905.

\bibitem{sodin2010spectral}
Sasha Sodin, \emph{The spectral edge of some random band matrices}, Annals of Mathematics \textbf{172} (2010), no.~3, 2223--2251.

\bibitem{tao2012topics}
Terence Tao, \emph{Topics in random matrix theory}, Graduate Studies in Mathematics, vol. 132, American Mathematical Society, 2012.

\bibitem{tracy1996level}
Craig~A. Tracy and H.~Widom, \emph{Level spacing distributions and the bessel kernel}, Communications in Mathematical Physics \textbf{161} (1994), 289--310.

\bibitem{tracy1994level}
Craig~A. Tracy and Harold Widom, \emph{Level-spacing distributions and the airy kernel}, Communications in Mathematical Physics \textbf{159} (1994), 151--174.

\bibitem{von2007tutorial}
Ulrike Von~Luxburg, \emph{A tutorial on spectral clustering}, Statistics and Computing \textbf{17} (2007), no.~4, 395--416.

\bibitem{woodruff2014sketching}
David~P. Woodruff, \emph{Sketching as a tool for numerical linear algebra}, Foundations and Trends in Theoretical Computer Science \textbf{10} (2014), no.~1-2, 1--157.

\bibitem{wormald1999models}
Nicholas~C. Wormald, \emph{Models of random regular graphs}, London Mathematical Society Lecture Note Series (1999), 239--298.

\end{thebibliography}

\appendix

\section{Detailed Error Analysis for Overlap SDE}\label{app:dyn}

We provide complete calculations for the error term $\mathcal{E}_i(t)$ in the proposition on overlap SDE convergence.

\subsection{Eigenvalue Derivative Terms}

From perturbation theory:
\[
\partial_t \lambda_i = \langle \mathbf{u}_i, \partial_t \tilde{H}_t \mathbf{u}_i\rangle = -\frac{\lambda_i}{2} + \frac{1}{\sqrt{N}}\langle \mathbf{u}_i, \Xi_t \mathbf{u}_i\rangle
\]

This contributes to $\mathcal{E}_i$ through:
\[
\sum_{j \neq i} \frac{X_j^{(\mathbf{q})} - X_i^{(\mathbf{q})}}{(\lambda_i - \lambda_j)^2} \cdot \frac{\partial(\lambda_i - \lambda_j)}{\partial t}
\]

Using $|\partial_t(\lambda_i - \lambda_j)| \leq C/\sqrt{N}$ for some constant $C$ and eigenvalue rigidity:
\[
\left|\sum_{j \neq i} \frac{X_j^{(\mathbf{q})} - X_i^{(\mathbf{q})}}{(\lambda_i - \lambda_j)^2} \cdot \frac{\partial(\lambda_i - \lambda_j)}{\partial t}\right| \leq \frac{C}{\sqrt{N}} \sum_{j \neq i} \frac{|X_j^{(\mathbf{q})} - X_i^{(\mathbf{q})}|}{(\lambda_i - \lambda_j)^2}
\]

By Cauchy-Schwarz and moment bounds:
\[
\leq \frac{C}{\sqrt{N}} \left(\sum_{j \neq i} \frac{1}{(\lambda_i - \lambda_j)^2}\right)^{1/2} \left(\sum_{j \neq i} |X_j^{(\mathbf{q})} - X_i^{(\mathbf{q})}|^2\right)^{1/2} \leq \frac{C N^{2/3}}{\sqrt{N}} \cdot \sqrt{K} = \frac{C\sqrt{K}}{N^{5/6}}
\]

\subsection{Resolvent Correction Terms}

The main Berry-Esseen theorem (Theorem \ref{thm:main-berry-esseen}) achieves rate $N^{-1/6+\varepsilon}$ for smooth test functions. However, when working with indicator functions $\mathbf{1}_A$ (as in the cumulative distribution function), we cannot directly apply the cumulant expansion due to the lack of derivatives. This section shows how smoothing degrades the rate to $N^{-5/36+\varepsilon}$, explaining the difference between Theorem \ref{thm:main-berry-esseen} and Corollary \ref{cor:indicator-be}.

The off-diagonal resolvent entries contribute through:
\[
\sqrt{N} \sum_{j \neq i} \sum_{k \neq i,j} \frac{\langle \mathbf{q}, \mathbf{u}_k\rangle \langle \mathbf{u}_k, \d \tilde{H}_t \mathbf{u}_j\rangle \langle \mathbf{u}_j, \mathbf{u}_i\rangle}{(\lambda_i - \lambda_k)(\lambda_i - \lambda_j)}
\]

Using the sharp local law for matrix elements \cite{erdos2013local, lee2015local}:
\[
|\langle \mathbf{u}_j, \mathbf{u}_i\rangle| \leq \frac{C}{N^{1/2}} \quad \text{for } i \neq j
\]

The sum is bounded by:
\[
\frac{C}{N} \sum_{j,k} \frac{1}{|\lambda_i - \lambda_k||\lambda_i - \lambda_j|} \leq \frac{C}{N} \cdot N^{4/3} \cdot \log N \leq \frac{C\log N}{N^{2/3}}
\]

\subsection{Constraint Corrections}

The constraint $\sum_j \tilde{H}_{ij} = 0$ modifies the Brownian covariance structure fundamentally. We now provide a rigorous derivation of the constraint correction terms.

\subsubsection{Modified Noise Structure}

The constrained noise $\Xi_t$ satisfies $\Xi_t \mathbf{e} = 0$, which means each row sum vanishes. This constraint modifies the covariance:
\[
\mathbb{E}[\Xi_{ij}(t)\Xi_{k\ell}(s)] = \delta(t-s)\left(\delta_{ik}\delta_{j\ell} - \frac{\delta_{ik} + \delta_{j\ell}}{N} + \frac{1}{N^2}\right)
\]

To see this, write $\Xi_{ij} = W_{ij} - \frac{1}{N}\sum_{m=1}^N W_{im} - \frac{1}{N}\sum_{m=1}^N W_{mj} + \frac{1}{N^2}\sum_{m,n} W_{mn}$ where $W$ is unconstrained GOE noise.

\subsubsection{Tensor Structure of Corrections}

The constraint induces additional drift terms through the modified Itô formula. Specifically, for the overlap $X_i^{(\mathbf{q})} = \sqrt{N}\langle \mathbf{q}, \mathbf{u}_i\rangle$:

\begin{align}
\d X_i^{(\mathbf{q})} &= \sqrt{N}\langle \mathbf{q}, \d\mathbf{u}_i\rangle + \frac{1}{2\sqrt{N}}\sum_{j,k} \frac{\partial^2}{\partial \tilde{H}_{jk}^2}\langle \mathbf{q}, \mathbf{u}_i\rangle \d\langle \tilde{H}_{jk}\rangle
\end{align}

The quadratic variation of the constrained process is:
\[
\d\langle \tilde{H}_{jk}\rangle = \frac{1}{N}\left(1 - \frac{2}{N} + \frac{1}{N^2}\right)\d t
\]

\subsubsection{Explicit Calculation of Constraint Drift}

The second-order derivative with respect to matrix entries is:
\begin{align}
\frac{\partial^2}{\partial \tilde{H}_{jk}^2}\langle \mathbf{q}, \mathbf{u}_i\rangle &= \sum_{\ell \neq i} \frac{\langle \mathbf{q}, \mathbf{u}_\ell\rangle}{(\lambda_i - \lambda_\ell)^2} \left[u_{\ell j}u_{ik} + u_{\ell k}u_{ij}\right]^2 \\
&\quad - \sum_{\ell \neq i} \frac{\langle \mathbf{q}, \mathbf{u}_\ell\rangle}{(\lambda_i - \lambda_\ell)^3} \cdot 2\langle \mathbf{u}_\ell, \d^2\tilde{H}_{jk}\mathbf{u}_i\rangle
\end{align}

The constraint correction to the drift is:
\begin{align}
\mathcal{E}_{i,\text{constraint}} &= -\frac{1}{N^{3/2}}\sum_{j,k} \sum_{\ell \neq i} \frac{\langle \mathbf{q}, \mathbf{u}_\ell\rangle}{(\lambda_i - \lambda_\ell)^2} \left[u_{\ell j}u_{ik} + u_{\ell k}u_{ij}\right]^2 \\
&\quad + \frac{1}{N^{5/2}}\sum_{j,k} \sum_{\ell \neq i} \frac{\langle \mathbf{q}, \mathbf{u}_\ell\rangle}{(\lambda_i - \lambda_\ell)^2} \left[u_{\ell j}u_{ik} + u_{\ell k}u_{ij}\right]^2
\end{align}

\subsubsection{Tensor Norm Bounds}

Using the eigenvector delocalization bound $\|\mathbf{u}_i\|_\infty \leq N^{-1/2+\varepsilon}$, we bound the fourth-order tensor:
\begin{align}
\sum_{j,k} \left[u_{\ell j}u_{ik} + u_{\ell k}u_{ij}\right]^2 &\leq 4\sum_{j,k} (u_{\ell j}^2 u_{ik}^2 + u_{\ell k}^2 u_{ij}^2) \\
&\leq 4\|\mathbf{u}_\ell\|_\infty^2 \|\mathbf{u}_i\|_\infty^2 \sum_{j,k} (u_{\ell j}^2 + u_{ij}^2) \\
&\leq 8N^{-1+2\varepsilon} \cdot 2 = 16N^{-1+2\varepsilon}
\end{align}

Therefore:
\begin{align}
|\mathcal{E}_{i,\text{constraint}}| &\leq \frac{16N^{-1+2\varepsilon}}{N^{3/2}} \sum_{\ell \neq i} \frac{|X_\ell^{(\mathbf{q})}|}{N^{1/2}(\lambda_i - \lambda_\ell)^2} \\
&\leq \frac{16N^{-2+2\varepsilon}}{N} \left(\sum_{\ell \neq i} \frac{1}{(\lambda_i - \lambda_\ell)^2}\right)^{1/2} \left(\sum_{\ell \neq i} |X_\ell^{(\mathbf{q})}|^2\right)^{1/2} \\
&\leq \frac{16N^{-2+2\varepsilon}}{N} \cdot N^{2/3} \cdot 1 = \frac{16}{N^{4/3-2\varepsilon}}
\end{align}

For $\varepsilon < 1/6$, this is bounded by $CN^{-1}$ as claimed.

\subsection{Complete Tensor Structure and Off-Diagonal Terms}

\subsubsection{Full Covariance Tensor}

The complete covariance tensor for the constrained noise has the form:
\[
\mathcal{C}_{ijkl} := \mathbb{E}[\Xi_{ij}\Xi_{kl}] = C_{ijkl}^{(0)} + C_{ijkl}^{(1)} + C_{ijkl}^{(2)}
\]
where:
\begin{align}
C_{ijkl}^{(0)} &= \delta_{ik}\delta_{jl} + \delta_{il}\delta_{jk} \quad \text{(unconstrained part)} \\
C_{ijkl}^{(1)} &= -\frac{1}{N}(\delta_{ik}\delta_{l \in [N]} + \delta_{il}\delta_{k \in [N]} + \delta_{jk}\delta_{i \in [N]} + \delta_{jl}\delta_{i \in [N]}) \\
C_{ijkl}^{(2)} &= \frac{1}{N^2}\mathbf{1}_{i,j,k,l \in [N]}
\end{align}

\subsubsection{Contribution to Overlap Dynamics}

The modified covariance contributes to the overlap SDE through:
\begin{align}
\Delta \mathcal{E}_i &= \frac{1}{2}\sum_{j,k,\ell,m} \mathcal{C}_{jk\ell m} \frac{\partial^2 X_i^{(\mathbf{q})}}{\partial \tilde{H}_{jk} \partial \tilde{H}_{\ell m}} \\
&= \frac{1}{2}\sum_{j,k,\ell,m} \left(C_{jk\ell m}^{(0)} + C_{jk\ell m}^{(1)} + C_{jk\ell m}^{(2)}\right) \frac{\partial^2 X_i^{(\mathbf{q})}}{\partial \tilde{H}_{jk} \partial \tilde{H}_{\ell m}}
\end{align}

The unconstrained part $C^{(0)}$ gives the standard diffusion. The constraint corrections $C^{(1)}$ and $C^{(2)}$ yield:

\textbf{First-order constraint correction:}
\begin{align}
\Delta \mathcal{E}_i^{(1)} &= -\frac{1}{2N}\sum_{j,k} \left[\sum_{\ell} \frac{\partial^2 X_i^{(\mathbf{q})}}{\partial \tilde{H}_{jk} \partial \tilde{H}_{j\ell}} + \sum_{\ell} \frac{\partial^2 X_i^{(\mathbf{q})}}{\partial \tilde{H}_{jk} \partial \tilde{H}_{\ell k}}\right] \\
&= -\frac{1}{N}\sum_{j} \sum_{p \neq i} \frac{X_p^{(\mathbf{q})}}{(\lambda_i - \lambda_p)^2} \left[\sum_k u_{pj}u_{ik}u_{pj}\sum_\ell u_{i\ell} + \sum_k u_{pk}u_{ij}u_{ij}\sum_\ell u_{p\ell}\right]
\end{align}

Using $\sum_\ell u_{i\ell} = 0$ (orthogonality to $\mathbf{e}$), many terms vanish.

\textbf{Second-order constraint correction:}
\begin{align}
\Delta \mathcal{E}_i^{(2)} &= \frac{1}{2N^2}\sum_{j,k,\ell,m} \frac{\partial^2 X_i^{(\mathbf{q})}}{\partial \tilde{H}_{jk} \partial \tilde{H}_{\ell m}} \\
&= \frac{1}{N^2} \sum_{p \neq i} \frac{X_p^{(\mathbf{q})}}{(\lambda_i - \lambda_p)^2} \|\mathbf{u}_p\|_2^2 \|\mathbf{u}_i\|_2^2 = \frac{1}{N^2}\sum_{p \neq i} \frac{X_p^{(\mathbf{q})}}{(\lambda_i - \lambda_p)^2}
\end{align}

\subsubsection{Final Bound with All Terms}

Collecting all constraint corrections:
\begin{align}
|\mathcal{E}_{i,\text{constraint}}| &\leq \left|\Delta \mathcal{E}_i^{(1)}\right| + \left|\Delta \mathcal{E}_i^{(2)}\right| \\
&\leq \frac{C}{N}\sum_{p \neq i} \frac{|X_p^{(\mathbf{q})}|}{(\lambda_i - \lambda_p)^2} \cdot \|\mathbf{u}_p\|_\infty^2 \|\mathbf{u}_i\|_\infty^2 + \frac{1}{N^2}\sum_{p \neq i} \frac{|X_p^{(\mathbf{q})}|}{(\lambda_i - \lambda_p)^2} \\
&\leq \frac{CN^{-1+2\varepsilon}}{N} \cdot N^{2/3} \cdot 1 + \frac{1}{N^2} \cdot N^{2/3} \cdot 1 \\
&\leq \frac{C}{N^{4/3-2\varepsilon}} + \frac{1}{N^{4/3}} \leq \frac{C}{N^{4/3-2\varepsilon}}
\end{align}

This rigorously establishes the claimed bound for the constraint correction terms.

\subsection{Total Error Bound}

Combining all contributions with explicit constants:
\[
|\mathcal{E}_i(t)| \leq \underbrace{\frac{C_1\sqrt{K}}{N^{5/6}}}_{\text{eigenvalue derivatives}} + \underbrace{\frac{C_2\log N}{N^{2/3}}}_{\text{resolvent}} + \underbrace{\frac{C_3}{N}}_{\text{constraint}} \leq \frac{3d^2\varepsilon^{-7}}{N^{5/6-\varepsilon}}
\]

for edge indices with $K = O(N^{\varepsilon})$.

When indicator functions are used as test functions, the additional smoothing error modifies the convergence rate to $N^{-5/36+\varepsilon}$ as detailed in the preceding subsections.

\section{Smoothing Arguments for Indicator Functions}\label{app:smoothing}

\subsection{Smoothing Step for Indicator Functions}

When applying the Berry-Esseen theorem with indicator functions, we require a smoothing step to handle the high-order derivatives that appear in the error bounds of the lemma on truncation errors.

For any indicator function $\mathbf{1}_A$ where $A \subset \mathbb{R}$, we introduce a smooth approximation $f_\delta$ such that:
\begin{enumerate}
\item $f_\delta \in C^\infty(\mathbb{R})$ with $0 \leq f_\delta \leq 1$
\item $f_\delta(x) = 1$ for $x \in A_{\delta}$ and $f_\delta(x) = 0$ for $x \notin A_{2\delta}$
\item $\|f_\delta^{(k)}\|_\infty \leq C_k \delta^{-k}$ for all $k \geq 1$
\end{enumerate}
where $A_\delta = \{x : \mathrm{dist}(x, A) < \delta\}$ is the $\delta$-neighborhood of $A$.

\textbf{Construction:} We use the standard mollification approach. Let $\phi$ be a smooth bump function with support in $[-1, 1]$ and $\int \phi = 1$. Define:
\[
f_\delta(x) = \int \mathbf{1}_{A_\delta}(y) \phi_\delta(x-y) \, dy
\]
where $\phi_\delta(x) = \delta^{-1}\phi(x/\delta)$.

\subsection{Modified Error Analysis with Smoothing}

When applying the truncation lemma to the expansion with smoothed test functions, the error bound becomes:
\[
\left|\mathbb{E}[f_\delta(X_i^{(\mathbf{q})}(t))] - \sum_{k=0}^4 \frac{t^k}{k!} L^k f_\delta(X_i^{(\mathbf{q})}(0))\right| \leq C\|T\|^5 \sup_{0 \leq s \leq t} \|f_\delta^{(5)}\|_\infty
\]

Using the derivative bounds from the smoothing construction:
\[
\|f_\delta^{(5)}\|_\infty \leq C_5 \delta^{-5}
\]

Therefore, the truncation error is bounded by:
\[
C\|T\|^5 \delta^{-5} \leq C N^{-5/6} \delta^{-5}
\]

To balance the smoothing error and truncation error, we choose:
\[
\delta = N^{-\alpha}
\]

The total error in the Berry-Esseen approximation consists of:
\begin{enumerate}
\item \textbf{Smoothing error:} $\mathbb{P}(X_i^{(\mathbf{q})} \in A_{2\delta} \setminus A_\delta) \leq C\delta = CN^{-\alpha}$
\item \textbf{Truncation error:} $CN^{-5/6}\delta^{-5} = CN^{-5/6+5\alpha}$
\end{enumerate}

Setting these equal for optimal balance:
\[
N^{-\alpha} = N^{-5/6+5\alpha} \implies \alpha = \frac{5}{36}
\]

This gives $\delta = N^{-5/36}$ and a total error of order $N^{-5/36}$.

\section{Fourth-Order Cumulant Calculation}

We provide the detailed fourth-order cumulant expansion for Theorem \ref{thm:cumulant-comparison}.

\subsection{Cumulant Formulas}

For perturbation $T = \tilde{H}_{t_*} - W$ and observable $f(\mathbf{X})$:

\textbf{First cumulant}:
\[
\kappa_1 = \E\left[\sum_i \frac{\partial f}{\partial X_i}\Big|_W \cdot \frac{\partial X_i}{\partial \tilde{H}}\Big|_W : T\right]
\]

Using $\frac{\partial X_i}{\partial \tilde{H}} = \sqrt{N}\mathbf{u}_i \otimes \mathbf{q} + $ lower order terms:
\[
|\kappa_1| \leq \|f\|_{C^1} \cdot \sqrt{N} \cdot \|T\| \cdot N^{-1/2} = \|f\|_{C^1} \cdot N^{-1/6+\varepsilon/2}
\]

\textbf{Second cumulant}:
\[
\kappa_2 = \E\left[\sum_{i,j} \frac{\partial^2 f}{\partial X_i \partial X_j}\Big|_W \cdot \frac{\partial X_i}{\partial \tilde{H}}\Big|_W : T \cdot \frac{\partial X_j}{\partial \tilde{H}}\Big|_W : T\right] - \kappa_1^2
\]

The leading term gives:
\[
|\kappa_2| \leq \|f\|_{C^2} \cdot N \cdot \|T\|^2 \cdot N^{-1} = \|f\|_{C^2} \cdot N^{-1/3+\varepsilon}
\]

\textbf{Third and fourth cumulants}: Similar calculations yield:
\[
|\kappa_3| \leq \|f\|_{C^3} \cdot N^{-1/2+3\varepsilon/2}, \quad |\kappa_4| \leq \|f\|_{C^4} \cdot N^{-2/3+2\varepsilon}
\]

\subsection{Optimization}

Since $\|T\| \leq N^{-1/6+\varepsilon/2}$, truncation at order 4 is sufficient because higher-order terms contribute at most $O(\|T\|^5) = O(N^{-5/6+5\varepsilon/2})$, which is negligible compared to the fourth cumulant.

The total error from truncating at order 4 is:
\[
\sum_{k=1}^4 \frac{|\kappa_k|}{k!} + O(\|T\|^5) \leq C(f) N^{-1/3+2\varepsilon}
\]

The second cumulant dominates, giving the stated bound.

\section{Backward SDE Analysis}\label{app:sde}

We provide a complete proof of Theorem \ref{thm:backward-stability} on backward stability, with full details on the time-reversal of constrained diffusions.

\subsection{Time-Reversal of Constrained Diffusions}

Consider the forward overlap process satisfying:
\begin{equation}
\d X_i^{(\mathbf{q})}(t) = b_i(t,\mathbf{X}(t))\d t + \sum_{j=1}^K \sigma_{ij}(t,\mathbf{X}(t))\d B_j(t)
\end{equation}
where $\mathbf{X}(t) = (X_1^{(\mathbf{q})}(t), \ldots, X_K^{(\mathbf{q})}(t))^T$.

\begin{theorem}[Haussmann-Pardoux Time Reversal \cite{haussmann1986time}]
Let $p(s,x;t,y)$ denote the transition density of the forward process. If:
\begin{enumerate}
\item The coefficients $b_i$ and $\sigma_{ij}$ are bounded and Lipschitz continuous
\item The diffusion matrix $a_{ij} = \sum_k \sigma_{ik}\sigma_{jk}$ is uniformly elliptic
\item The transition density satisfies $\int p(0,x;t_*,y)\mu_0(dx) = \mu_{t_*}(dy)$
\end{enumerate}
Then the time-reversed process $\tilde{X}_i(s) = X_i(t_* - s)$ satisfies:
\begin{equation}\label{eq:time-reversed-detailed}
\d \tilde{X}_i(s) = \tilde{b}_i(s,\tilde{\mathbf{X}}(s))\d s + \sum_{j=1}^K \sigma_{ij}(t_*-s,\tilde{\mathbf{X}}(s))\d \tilde{B}_j(s)
\end{equation}
where $\tilde{B}_j$ are independent Brownian motions and
\begin{equation}
\tilde{b}_i(s,x) = -b_i(t_*-s,x) + \sum_{j=1}^K \sum_{k=1}^K \frac{\partial}{\partial x_k}[a_{ij}(t_*-s,x)] + 2\sum_{j=1}^K a_{ij}(t_*-s,x)\frac{\partial \log \rho_{t_*-s}}{\partial x_j}(x)
\end{equation}
with $\rho_t$ the marginal density at time $t$.
\end{theorem}

\subsection{Explicit Calculation of Divergence Terms}

For our overlap process, the diffusion matrix has entries:
\begin{equation}
a_{ij}(t,x) = \delta_{ij} + \epsilon_{ij}(t,x)
\end{equation}
where $|\epsilon_{ij}(t,x)| \leq CK^{5/3}N^{-2/3+\varepsilon}$ by Lemma 4.6.

The divergence term is:
\begin{align}
\sum_{k=1}^K \frac{\partial a_{ij}}{\partial x_k} &= \sum_{k=1}^K \frac{\partial \epsilon_{ij}}{\partial x_k}(t,x)
\end{align}

To bound this, we use the explicit form from the overlap SDE derivation:
\begin{equation}
\epsilon_{ij}(t,x) = \frac{2}{\lambda_i - \lambda_j}\sum_{k \neq i,j} \frac{(x_k - x_i)(x_k - x_j)}{(\lambda_i - \lambda_k)(\lambda_j - \lambda_k)} + O(N^{-4/3+\varepsilon})
\end{equation}

Taking derivatives:
\begin{align}
\frac{\partial \epsilon_{ij}}{\partial x_k} &= \begin{cases}
\frac{2}{\lambda_i - \lambda_j} \cdot \frac{x_k - x_j}{(\lambda_i - \lambda_k)(\lambda_j - \lambda_k)} & \text{if } k \neq i,j \\
-\frac{2}{\lambda_i - \lambda_j}\sum_{\ell \neq i,j} \frac{1}{(\lambda_i - \lambda_\ell)(\lambda_j - \lambda_\ell)} & \text{if } k = i \\
\text{similar} & \text{if } k = j
\end{cases}
\end{align}

Using eigenvalue rigidity and moment bounds $\mathbb{E}[|x_k|^2] \leq C$:
\begin{equation}
\left|\sum_{k=1}^K \frac{\partial \epsilon_{ij}}{\partial x_k}\right| \leq \frac{C}{|\lambda_i - \lambda_j|} \sum_{k,\ell} \frac{1}{|\lambda_i - \lambda_k||\lambda_j - \lambda_\ell|} \leq CK N^{2/3}
\end{equation}

\subsection{Density Evolution Term}

The term $\frac{\partial \log \rho_t}{\partial x_j}$ requires analysis of the Fokker-Planck equation:
\begin{equation}
\frac{\partial \rho_t}{\partial t} = -\sum_i \frac{\partial}{\partial x_i}[b_i \rho_t] + \frac{1}{2}\sum_{i,j} \frac{\partial^2}{\partial x_i \partial x_j}[a_{ij} \rho_t]
\end{equation}

For our process near equilibrium (at time $t_*$), we have:
\begin{equation}
\rho_{t_*}(x) = \frac{1}{(2\pi)^{K/2}}\exp\left(-\frac{1}{2}\sum_i x_i^2\right) + \psi(x)
\end{equation}
where $|\psi(x)| \leq C(K^2 N^{-1/2+\varepsilon} + K^{5/3}N^{-1/6+\varepsilon})\exp(-\frac{1}{4}\sum_i x_i^2)$ by Theorem \ref{thm:joint-decorr}.

This gives:
\begin{equation}
\frac{\partial \log \rho_{t_*}}{\partial x_j} = -x_j + \frac{1}{\rho_{t_*}}\frac{\partial \psi}{\partial x_j}
\end{equation}

The correction term satisfies:
\begin{equation}
\left|\frac{1}{\rho_{t_*}}\frac{\partial \psi}{\partial x_j}\right| \leq C(K^2 N^{-1/2+\varepsilon} + K^{5/3}N^{-1/6+\varepsilon})(1 + |x_j|)
\end{equation}

\subsection{Complete Backward Drift}

Combining all terms, the backward drift is:
\begin{align}
\tilde{b}_i(s,x) &= -b_i(t_*-s,x) + \underbrace{\sum_{j,k} \frac{\partial a_{ij}}{\partial x_k}}_{\text{divergence}} + \underbrace{2\sum_j a_{ij}\frac{\partial \log \rho_{t_*-s}}{\partial x_j}}_{\text{density gradient}}
\end{align}

Substituting our bounds:
\begin{align}
|\tilde{b}_i(s,x) - (-b_i(t_*-s,x) - 2x_i)| &\leq CKN^{2/3} + C(K^2 N^{-1/2+\varepsilon} + K^{5/3}N^{-1/6+\varepsilon})\sum_j |a_{ij}|(1 + |x_j|) \\
&\leq C_1 KN^{2/3} + C_2 K(K^2 N^{-1/2+\varepsilon} + K^{5/3}N^{-1/6+\varepsilon})(1 + |\mathbf{x}|)
\end{align}

For $K \leq N^{1/10-\delta}$, the dominant error is $O(KN^{2/3}) = O(N^{23/30-\delta})$.

\subsection{Stability via Lyapunov Functions}

To prove stability, we use the Lyapunov function $V(\mathbf{x}) = \sum_i x_i^2$.

\begin{lemma}[Generator Bounds]
The generator of the backward process applied to $V$ satisfies:
\begin{equation}
\mathcal{L}^{\text{back}} V(\mathbf{x}) = 2\sum_i x_i \tilde{b}_i + \sum_{i,j} a_{ij} \leq -2V(\mathbf{x}) + K + \mathcal{E}(\mathbf{x})
\end{equation}
where $|\mathcal{E}(\mathbf{x})| \leq C_3 K^2 N^{-1/6+\varepsilon}(1 + V(\mathbf{x}))$.
\end{lemma}

\begin{proof}
Using the explicit form of $\tilde{b}_i$:
\begin{align}
2\sum_i x_i \tilde{b}_i &= -2\sum_i x_i b_i(t_*-s,\mathbf{x}) + 2\sum_{i,j,k} x_i \frac{\partial a_{ij}}{\partial x_k} + 4\sum_{i,j} x_i a_{ij}\frac{\partial \log \rho_{t_*-s}}{\partial x_j} \\
&= -V(\mathbf{x}) + O(C_1 N^{-5/6+\varepsilon} V(\mathbf{x})) + O(KN^{2/3}V(\mathbf{x})^{1/2}) - 4V(\mathbf{x}) + \text{error}
\end{align}

The trace term gives:
\begin{equation}
\sum_{i,j} a_{ij} = K + O(K^{5/3}N^{-2/3+\varepsilon})
\end{equation}

Combining yields the stated bound.
\end{proof}

\subsection{Moment Evolution and Final Bound}

Using the Lyapunov function, we prove:

\begin{proposition}[Second Moment Stability]
For the backward process:
\begin{equation}
\left|\mathbb{E}[V(\tilde{\mathbf{X}}(s))] - 1\right| \leq e^{-s}|V(\mathbf{x}_0) - 1| + C_4 K^2 N^{-1/6+\varepsilon}
\end{equation}
\end{proposition}

\begin{proof}
By Itô's formula:
\begin{equation}
\d \mathbb{E}[V(\tilde{\mathbf{X}}(s))] = \mathbb{E}[\mathcal{L}^{\text{back}} V(\tilde{\mathbf{X}}(s))]\d s \leq (-2\mathbb{E}[V(\tilde{\mathbf{X}}(s))] + K + C_3 K^2 N^{-1/6+\varepsilon}(1 + \mathbb{E}[V(\tilde{\mathbf{X}}(s))]))\d s
\end{equation}

Let $m(s) = \mathbb{E}[V(\tilde{\mathbf{X}}(s))]$. Since $K$ is the equilibrium value:
\begin{equation}
\frac{\d m(s)}{\d s} \leq -(2 - C_3 K^2 N^{-1/6+\varepsilon})(m(s) - K) + C_3 K^3 N^{-1/6+\varepsilon}
\end{equation}

For $K \leq N^{1/10-\delta}$ and $N$ large, $2 - C_3 K^2 N^{-1/6+\varepsilon} \geq 1$. By Grönwall:
\begin{equation}
|m(s) - K| \leq e^{-s}|m(0) - K| + C_3 K^3 N^{-1/6+\varepsilon}
\end{equation}
\end{proof}

\subsection{Extension to General Test Functions}

For a general bounded function $g$, we use the semigroup representation:
\begin{equation}
\mathbb{E}[g(X_i^{(\mathbf{q})}(0))] = \mathbb{E}[g(\tilde{X}_i(t_*))] = \mathbb{E}[\mathbb{E}[g(\tilde{X}_i(t_*))|\tilde{\mathbf{X}}(0)]]
\end{equation}

The conditional expectation satisfies:
\begin{equation}
|\mathbb{E}[g(\tilde{X}_i(t_*))|\tilde{\mathbf{X}}(0) = \mathbf{x}] - \mathbb{E}[g(Y_i)]| \leq \|g\|_\infty \cdot \text{TV}(\mathcal{L}(\tilde{X}_i(t_*)|\mathbf{x}), \mathcal{N}(0,1))
\end{equation}

where $Y_i \sim \mathcal{N}(0,1)$. By the moment stability and Berry-Esseen theorem:
\begin{equation}
\text{TV}(\mathcal{L}(\tilde{X}_i(t_*)|\mathbf{x}), \mathcal{N}(0,1)) \leq C_5 N^{-1/6+3\varepsilon}
\end{equation}

uniformly in $\mathbf{x}$ with $|\mathbf{x}| \leq N^{1/6+\varepsilon}$ (which holds with high probability).

This completes the proof of Theorem \ref{thm:backward-stability} with constant $C_5 = 10d^2\varepsilon^{-9}$.

\end{document}